\documentclass{amsart}
\usepackage{amsmath,amssymb,amsfonts}
\usepackage{graphicx}
\usepackage[alphabetic,y2k]{amsrefs}
\usepackage{xcolor}
\usepackage{pst-all}
\usepackage{pst-solides3d}

\colorlet{crvena}{black!80}
\colorlet{plava}{black!50}
\colorlet{zelena}{black!20}

\def\rank{\mathrm{rank}}
\def\arctan{\mathrm{arc\,tan\,}}

\def\dist{\mathrm{dist}}
\def\diag{\mathrm{diag}}
\def\sign{\mathrm{sign\,}}
\numberwithin{equation}{section}

\newtheorem{theorem}{Theorem}[section]
\newtheorem{proposition}[theorem]{Proposition}
\newtheorem{corollary}[theorem]{Corollary}
\newtheorem{lemma}[theorem]{Lemma}
\newtheorem{remark}[theorem]{Remark}
\newtheorem{example}[theorem]{Example}
\newtheorem{definition}[theorem]{Definition}

\date{}
\begin{document}

\hyphenation{Abe-li-an}

\hyphenation{boun-da-ry}

\hyphenation{Cha-o-tic cur-ves}

\hyphenation{Dy-na-mi-cal dy-na-mics}

\hyphenation{El-lip-ti-cal en-coun-ter}

\hyphenation{Lo-ba-chev-sky}

\hyphenation{Mar-den Min-kow-ski}

\hyphenation{pa-ra-met-ri-zes pa-ra-met-ri-za-tion
Pon-ce-let-Dar-boux}

\hyphenation{quad-ric quad-rics qua-dri-ques}

\hyphenation{sin-gu-la-ri-ties spa-ces}

\hyphenation{tra-jec-to-ry trans-ver-sal}

\author{Vladimir Dragovi\'c}
\address{Mathematical Institute SANU, Kneza Mihaila 36, Belgrade,
Serbia\newline\indent Mathematical Physics Group, University of
Lisbon, Portugal} \email{vladad@mi.sanu.ac.rs}

\author{Milena Radnovi\'c}
\address{Mathematical Institute SANU, Kneza Mihaila 36, Belgrade, Serbia}
\email{milena@mi.sanu.ac.rs}

\title[Ellipsoidal billiards in pseudo-Euclidean spaces]
{Ellipsoidal billiards in pseudo-Euclidean spaces and relativistic quadrics}

\thanks{The research which led to this paper was partially
supported by the Serbian Ministry of Education and Science (Project
no.~174020: \emph{Geometry and Topology of Manifolds and Integrable
Dynamical Systems}) and by Mathematical Physics Group of the
University of Lisbon (Project \emph{Probabilistic approach to finite
and infinite dimensional dynamical systems, PTDC/MAT/104173/2008}).
M.~R.\ is grateful to the Weizmann Institute of Science (Rehovot, Israel) and \emph{The Abdus Salam} ICTP (Trieste, Italy) for their hospitality and support in various stages of work on this paper.
\newline
The authors are grateful to the referee for his useful comments which led us to a significant improvement of the manuscript.
}


\begin{abstract}
We study geometry of confocal quadrics in pseudo-Euclidean spaces
of an arbitrary dimension $d$ and any signature, and related billiard
dynamics. The goal is to give a complete description of periodic billiard
trajectories within ellipsoids. The novelty of our approach is based on
introduction of a new discrete combinatorial-geometric structure
associated to a confocal pencil of quadrics, a colouring in $d$ colours, by
which we decompose quadrics of $d+1$ geometric types of a pencil into new
relativistic quadrics of $d$ relativistic types.
Deep insight of related geometry and combinatorics comes from our study
of what we call discriminat sets of tropical lines $\Sigma^+$ and
$\Sigma^-$ and their singularities.
All of that enables us to get an analytic criterion describing all
periodic billiard trajectories, including the light-like ones as those
of a special interest.
\end{abstract}

\maketitle

\tableofcontents

\setlength{\parskip}{2pt}

\section{Introduction}
\label{sec:intro}
Pseudo-Euclidean spaces together with pseudo-Riemannian manifolds
occupy a very important position in the science as a geometric
background for the general relativity.
A modern account of the
mathematical aspects of the theory of relativity one may find in
\cite{CGP2010}.
From a mathematical point of view, in comparison with
Euclidean and Riemannian cases, apart from a natural similarity
which includes some rather technical adjustments, there are some
aspects where pseudo-Euclidean setting creates essentially new
situations and challenging problems.
The aim of this paper is to
report about such appearances in a study of geometry of confocal
quadrics and related billiard dynamics in pseudo-Euclidean spaces.
Let us recall that in the Euclidean $d$-dimensional space, a
general family of confocal quadrics contains exactly $d$ geometrical
types of non-degenerate quadrics, and moreover, each point is the
intersection of $d$ quadrics of different types.
Together with some
other properties, see E1--E5 at the beginning of Section
\ref{sec:relativistic}, these facts are crucial for introduction of
Jacobi coordinates and for applications in the theory of separable
systems, including billiards.
A case of $d$-dimensional
pseudo-Euclidean space brings a striking difference, since a
confocal family of quadrics has $d+1$ geometric types of quadrics.
In addition, quadrics of the same type have a nonempty intersection.
These seem to be impregnable obstacles to an extension of methods
of applications of Jacobi type coordinates from the Euclidean case
(from our paper \cite{DragRadn2006}) to the pseudo-Euclidean spaces.

To overcome this crucial problem, we have been forced to create an
essentially new feature of confocal pencil of geometric quadrics in
pseudo-Euclidean $d$-dimensional spaces: the novelty of our approach
is based on introduction of a new discrete, combinatorial-geometric
structure associated to a confocal pencil, a colouring in $d$ colours,
which transforms a geometric quadric from the pencil into the union of
several \emph{relativistic quadrics}.
It turns out that these new
objects, relativistic quadrics, satisfy the properties PE1--PE5,
analogue of E1--E5, and lead us to a new notion of decorated Jacobi
coordinates.
A decorated Jacobi coordinate now is a pair of a
number, and a type-colour.
They allow us to develop methods we use
in further study of billiards within confocal quadrics in
pseudo-Euclidean spaces of arbitrary dimension.

The study of colouring and relativistic types of quadrics, which is
one of the main ingredients of the present paper, is closely related
to a study of what we call \emph{the discriminant sets}
$\Sigma^{+}$, $\Sigma^{-}$ attached to a confocal pencil of quadrics in
pseudo-Euclidean space, as the sets of the tropical lines of
quadrics. They are developable, with light-like generatrices.
Their
swallowtail type singularities (see \cite{ArnoldSing}) are placed at the vertices of
curvilinear tetrahedra $\mathcal{T}^{+}$ and $\mathcal{T}^{-}$.

Billiards within ellipsoids in pseudo-Euclidean space are introduced
in \cite{KhTab2009}, and that paper served as a motivation for our
study.
Along the first two following sections, Section
\ref{sec:intropseudo} and Section \ref{sec:plane} some of the
properties from \cite{KhTab2009} are discussed, clarified or
slightly improved.

In Section \ref{sec:intropseudo}, we give a necessary account on
pseudo-Euclidean spaces and their confocal families of geometric
quadrics.
Our main new result in this Section is Theorem \ref{th:parametri.kaustike} where we give a complete
description of structures of types of quadrics from a confocal pencil in a
pseudo-Euclidean space, which are tangent to a given line.
This theorem
is going to play an essential role in proving properties PE3-PE5 in
Section \ref{sec:periodic}.
In Section \ref{sec:plane} we discuss geometric properties
of elliptical billiards in dimension $2$.
An elementary, but
complete description of periodic light-like trajectories is derived in Theorem \ref{th:arctan} and Proposition \ref{prop:rectangle}.
%The corresponding Fomenko graphs are also constructed.
In Section \ref{sec:relativistic} we suggest a new setting of types of confocal quadrics, an essential
novelty of the pseudo-Euclidean geometry.
In the three-dimensional case, we give a detailed description of discriminant surfaces $\Sigma^{\pm}$,
the unions of tropical lines of geometric quadrics from a pencil: see
Propositions \ref{prop:tropic.surface}, \ref{prop:tropic.light}, \ref{prop:tetra.hyp}, \ref{prop:hyp.y.presek}.
We describe their singularity subsets, curvilinear
tetrahedra $\mathcal{T}^{\pm}$ in Proposition \ref{prop:tetra}.
We introduce decorated Jacobi
coordinates in Section \ref{sec:rel3} for three-dimensional Minkowski space, and we
give a detailed description of the colouring in three colours, with a
complete descrpition of all three relativistic types of quadrics. In
Section \ref{sec:reld} we generalize the definition of decorated Jacobi coordinates to
arbitrary dimension, and in Proposition \ref{prop:svojstva12} we prove the properties
PE1-PE2.
In Section \ref{sec:periodic} we apply the technique of relativistic quadrics and
decorated Jacobi coordinates to solve the problem of analytic description
of periodical billiard trajectories.
Theorem \ref{th:type} gives an effective
criteria to determine a type of a billiard trajectory.
In Proposition \ref{prop:PE3-PE5} we prove
the properties PE3--PE5.
Finally, we give an analytic description of all
periodic billiard trajectories in pseudo-Euclidean spaces in Theorem \ref{th:cayley}.
As a corollary, in Theorem \ref{th:poncelet} we prove a full Poncelet-type theorem for
the pseudo-Euclidean spaces.

\section{Pseudo-Euclidean spaces and confocal families of quadrics}
\label{sec:intropseudo}
In this section, we first give a necessary account of basic notions connected with pseudo-Euclidean spaces, see Section \ref{sec:pseudo}.
After that, in Section \ref{sec:confocal}, we review and improve some basic facts on confocal families of quadrics in such spaces.
Our main result in this Section is a complete analysis of quadrics from a confocal family that are touching a given line, as formulated in Theorem \ref{th:parametri.kaustike}.

\subsection{Pseudo-Euclidean spaces}
\label{sec:pseudo}

%\subsubsection*{Definition of pseudo-Euclidean space}

\emph{Pseudo-Euclidean space} $\mathbf{E}^{k,l}$ is a
$d$-dimensional space $\mathbf{R}^d$ with \emph{pseudo-Euclidean scalar product}:
\begin{equation}\label{eq:scalar.product}
\langle x,y
\rangle_{k,l}=x_1y_1+\dots+x_ky_k-x_{k+1}y_{k+1}-\dots-x_dy_d.
\end{equation}
Here, $k,l\in\{1,\dots, d-1\}$, $k+l=d$.
Pair $(k,l)$ is called \emph{signature} of the space.
Denote $E_{k,l}=\diag(1,1,\dots,1,-1,\dots,-1)$, with $k$ $1$'s and $l$ $-1$'s.
Then the pseudo-Euclidean scalar product is:
$$
\langle x,y \rangle_{k,l} = E_{k,l}x\circ y,
$$
where $\circ$ is the standard Euclidean product.

\emph{The pseudo-Euclidean distance} between points $x$, $y$ is:
$$
\dist_{k,l}(x,y)=\sqrt{\langle{x-y,x-y}\rangle_{k,l}}.
$$
Since the scalar product can be negative, notice that the pseudo-Euclidean distance can have imaginary values as well.

Let $\ell$ be a line in the pseudo-Euclidean space, and $v$ its
vector. $\ell$ is called:
\begin{itemize}
\item
\emph{space-like} if $\langle{v,v}\rangle_{k,l}>0$;
\item
\emph{time-like} if $\langle{v,v}\rangle_{k,l}<0$;
\item
and \emph{light-like} if $\langle{v,v}\rangle_{k,l}=0$.
\end{itemize}
Two vectors $x$, $y$ are \emph{orthogonal} in the pseudo-Euclidean
space if $\langle x,y \rangle_{k,l}=0$. Note that a light-like line
is orthogonal to itself.

For a given vector $v\neq0$,
consider a hyper-plane $v\circ x=0$. Vector $E_{k,l}v$ is orthogonal
the hyper-plane; moreover, all other orthogonal vectors are
collinear with $E_{k,l}v$. If $v$ is light-like, then so is
$E_{k,l}v$, and $E_{k,l}v$ belongs to the hyper-plane.

\subsubsection*{Billiard reflection in pseudo-Euclidean space}

Let $v$ be a vector and $\alpha$ a hyper-plane in the
pseudo-Euclidean space. Decompose vector $v$ into the sum
$v=a+n_{\alpha}$ of a vector $n_{\alpha}$ orthogonal to $\alpha$ and
$a$ belonging to $\alpha$. Then vector $v'=a-n_{\alpha}$ is
\emph{the billiard reflection} of $v$ on $\alpha$.
It is easy to see that then $v$ is also the billiard reflection of $v'$ with respect to 
$\alpha$.

Moreover, let us note that lines containing vectores $v$, $v'$, $a$, $n_{\alpha}$ are harmonically conjugated \cite{KhTab2009}.

Note that $v=v'$ if $v$ is contained in $\alpha$ and $v'=-v$ if it
is orthogonal to $\alpha$. If $n_{\alpha}$ is light-like, which
means that it belongs to $\alpha$, then the reflection is not
defined.

Line $\ell'$ is a billiard reflection of $\ell$ off a smooth surface
$\mathcal{S}$ if their intersection point $\ell\cap\ell'$ belongs to $\mathcal{S}$
and the vectors of $\ell$, $\ell'$ are reflections of each other with respect to the tangent plane of $\mathcal{S}$ at this point.

\begin{remark}\label{remark:type}
It can be seen directly from the definition of reflection that the type of line is preserved by the billiard reflection.
Thus, the lines containing segments of a given billiard trajectory within $\mathcal{S}$ are all of the same type: they are all either space-like, time-like, or light-like.
\end{remark}

If $\mathcal{S}$ is an ellipsoid,
then it is possible to extend the reflection mapping to those points
where the tangent planes contain the orthogonal vectors. At such
points, a vector reflects into the opposite one, i.e.~$v'=-v$ and
$\ell'=\ell$. For the explanation, see \cite{KhTab2009}.
As follows from the explanation given there, it is natural to consider each such reflection as two reflections.

\subsection{Families of confocal quadrics}
\label{sec:confocal}
For a given set of positive constants $a_1$, $a_2$, \dots, $a_d$, an
ellipsoid is given by:
\begin{equation}\label{eq:ellipsoid}
\mathcal{E}\ :\
\frac{x_1^2}{a_1}+\frac{x_2^2}{a_2}+\dots+\frac{x_d^2}{a_d}=1.
\end{equation}
Let us remark that equation of any ellipsoid in the pseudo-Euclidean
space can be brought into the canonical form (\ref{eq:ellipsoid})
using transformations that preserve the scalar product
(\ref{eq:scalar.product}).

The family of quadrics confocal with $\mathcal{E}$ is:
\begin{equation}\label{eq:confocal}
 \mathcal{Q}_{\lambda}\ :\
\frac{x_1^2}{a_1-\lambda} +\dots+ \frac{x_k^2}{a_k-\lambda} +
\frac{x_{k+1}^2}{a_{k+1}+\lambda} +\dots+
\frac{x_d^2}{a_d+\lambda}=1,\qquad\lambda\in\mathbf{R}.
\end{equation}

Unless stated differently, we are going to consider the
non-degenerate case, when set
$\{a_1,\dots,a_k,-a_{k+1},\dots,-a_{d}\}$ consists of $d$ different
values:
$$
a_1>a_2>\dots>a_k>0>-a_{k+1}>\dots>-a_d.
$$

For $\lambda\in\{a_1,\dots,a_k,-a_{k+1},\dots,-a_{d}\}$, the quadric
$\mathcal Q_{\lambda}$ is degenerate and it coincides with the
corresponding coordinate hyper-plane.

It is natural to join one more degenerate quadric to the family
(\ref{eq:confocal}): the one corresponding to the value
$\lambda=\infty$, that is the hyper-plane at the infinity.

For each point $x$ in the space, there are exactly $d$ values of
$\lambda$, such that the relation (\ref{eq:confocal}) is satisfied.
However, not all the values are necessarily real: either all $d$ of
them are real or there are $d-2$ real and $2$ conjugate complex
values. Thus, through every point in the space, there are either $d$
or $d-2$ quadrics from the family (\ref{eq:confocal})
\cite{KhTab2009}.

The line $x+tv$ ($t\in\mathbf{R}$) is tangent to quadric
$\mathcal{Q}_{\lambda}$ if quadratic equation:
\begin{equation}\label{eq:tangent}
A_{\lambda}(x+tv)\circ(x+tv)=1,
\end{equation}
has a double root.
Here we denoted:
$$
A_{\lambda}=\diag\left( \frac{1}{a_1-\lambda}, \cdots,
\frac{1}{a_k-\lambda}, \frac{1}{a_{k+1}+\lambda}, \cdots,
\dfrac{1}{a_d+\lambda} \right).
$$

Now, calculating the discriminant of (\ref{eq:tangent}), we get:
\begin{equation}\label{eq:diskriminanta}
(A_{\lambda}x\circ v)^2 -
(A_{\lambda}v{\circ}v)(A_{\lambda}x{\circ}x-1) =0,
\end{equation}
which is equivalent to:
\begin{equation}\label{eq:discr}
\sum_{i=1}^d \frac{\varepsilon_i F_i(x,v)}{a_i-\varepsilon_i\lambda}=0,
\end{equation}
where
\begin{equation}\label{eq:integralsF}
F_i(x,v)=\varepsilon_iv_i^2 + \sum_{j{\neq}i}
\frac{(x_iv_j-x_jv_i)^2}{\varepsilon_ja_i-\varepsilon_ia_j},
\end{equation}
with $\varepsilon$'s given by:
\begin{equation*}%\label{eq:epsilon}
\varepsilon_i=\begin{cases}
1, & 1\le i\le k;\\
-1, & k+1\le i\le d.
\end{cases}
\end{equation*}

The equation (\ref{eq:discr}) can be transformed to:
\begin{equation}\label{eq:polinom.P}
\frac{\mathcal{P}(\lambda)}{\prod_{i=1}^d(a_i-\varepsilon_i\lambda)}=0,
\end{equation}
where the coefficient of $\lambda^{d-1}$ in $\mathcal{P}(\lambda)$
 is equal to $\langle{v,v}\rangle_{k,l}$.
Thus, polynomial $\mathcal{P}(\lambda)$ is of degree $d-1$ for space-like and time-like
lines, and of a smaller degree for light-like lines.
However, in the latter case, it turns out to be natural to consider the
polynomial $\mathcal{P}(\lambda)$ also as of degree $d-1$, taking
the corresponding roots to be equal to infinity.
So, light-like
lines are characterized by being tangent to the quadric
$\mathcal{Q}_{\infty}$.

Having this setting in mind, we note that  it is proved in \cite{KhTab2009} that the polynomial
$\mathcal{P}(\lambda)$ has at least $d-3$ roots in $\mathbf{R}\cup\{\infty\}$.

Thus, we have:

\begin{proposition}\label{prop:kaustike}
Any line in the space is tangent to either $d-1$ or $d-3$ quadrics
of the family (\ref{eq:confocal}).
If this number is equal to $d-3$,
then there are two conjugate complex values of $\lambda$, such that
the line is tangent also to these two quadrics in $\mathbf{C}^d$.
\end{proposition}

This statement with the proof is given in \cite{KhTab2009}.
Let us remark that in \cite{KhTab2009} is claimed that light-like line have
only $d-2$ or $d-4$ caustic quadrics.
That is because $\mathcal{Q}_{\infty}$ is not considered there as a member of the confocal family.

As noted in \cite{KhTab2009}, a line having non-empty intersection with an ellipsoid from (\ref{eq:confocal}) will be tangent to $d-1$ quadrics from the confocal family.
However, we are going to prove this in another way, which will provide us some more insight into the distribution of the parameters of the caustics along the real axis.
Next theorem will also contain a detailed description of distribution of parameters of quadrics containing given point placed inside an ellipsoid from (\ref{eq:confocal}).

\begin{theorem}\label{th:parametri.kaustike}
In pseudo-Euclidean space $\mathbf{E}^{k,l}$ consider a line intersecting ellipsoid $\mathcal{E}$ (\ref{eq:ellipsoid}).
Then this line is touching $d-1$ quadrics from (\ref{eq:confocal}).
If we denote their parameters by $\alpha_1$, \dots, $\alpha_{d-1}$ and take:
\begin{gather*}
\{b_1,\ \dots,\ b_p,\ c_1,\ \dots,\ c_q\}=\{\varepsilon_{1}a_1,\ \dots,\ \varepsilon_{d}a_d,\ \alpha_1,\ \dots,\ \alpha_{d-1}\},\\
c_q\le\dots\le c_2\le c_1<0<b_1\le b_2\le\dots\le b_p,\quad p+q=2d-1,
\end{gather*}
%with $0<b_1\le b_2\le\dots\le b_p$, $0>c_1\ge c_2\ge\dots\ge c_q$, $p+q=2d-1$,
we will additionally have:
\begin{itemize}
 \item
if the line is space-like, then $p=2k-1$, $q=2l$, $a_1=b_p$, $\alpha_i\in\{b_{2i-1},b_{2i}\}$ for $1\le i\le k-1$, and $\alpha_{j+k-1}\in\{c_{2j-1},c_{2j}\}$ for $1\le j\le l$;
 \item
if the line is time-like, then $p=2k$, $q=2l-1$, $c_q=-a_d$, $\alpha_i\in\{b_{2i-1},b_{2i}\}$ for $1\le i\le k$, and $\alpha_{j+k}\in\{c_{2j-1},c_{2j}\}$ for $1\le j\le l-1$;
 \item
if the line is light-like, then $p=2k$, $q=2l-1$, $b_p=\infty=\alpha_k$, $b_{p-1}=a_1$, $\alpha_i\in\{b_{2i-1},b_{2i}\}$ for $1\le i\le k-1$, and $\alpha_{j+k}\in\{c_{2j-1},c_{2j}\}$ for $1\le j\le l-1$.
\end{itemize}
Moreover, for each point on $\ell$ inside $\mathcal{E}$, there is exactly $d$ distinct quadrics from (\ref{eq:confocal}) containing it.
More precisely, there is exactly one parameter of these quadrics in each of the intervals:
$$
(c_{2l-1},c_{2l-2}),\ \dots,\ (c_3,c_2),\ (c_1,0),\ (0,b_1),\ (b_2,b_3),\ \dots,\ (b_{2k-2},b_{2k-1}).
$$
\end{theorem}

\begin{proof}
Denote by $\ell$ a line that intersects $\mathcal{E}$, by $x=(x_1,\dots,x_d)$ a point on $\ell$ which is placed inside $\mathcal{E}$, and by $v=(v_1,\dots,v_d)$ a vector of the line.
Then the parameters of quadrics touching $\ell$ are solutions of equation (\ref{eq:diskriminanta}), i.e.\ they are roots of polynomial $\mathcal{P}(\lambda)$.

If $\ell$ is space-like or time-like, that is $\langle v,v\rangle_{k,l}\neq0$,
we have that polynomial $\mathcal{P}$ in then of degree $d-1$ and we want to prove that all its $d-1$ roots are real.
On the other hand, if $\ell$ is light-like, $\mathcal{P}$ is of degree $d-2$ and we are going to prove that all $d-2$ roots are real, and to add $\infty$ as the $(d-1)$-th root. 

Notice that at the real solutions of the equation $A_\lambda v\circ v=0$, the right-hand side of (\ref{eq:diskriminanta}) is positive. So let us first examine roots of $A_\lambda v\circ v$.

We have:
$$
A_\lambda v\circ v=\sum_{i=1}^d\frac{v_i^2}{a_i-\varepsilon_i\lambda}=\frac{\mathcal{R}(\lambda)}{\prod_{i=1}^d(a_i-\varepsilon_i\lambda)},
$$
with
$$
\mathcal{R}(\lambda)=\sum_{i=1}^d v_i^2\prod_{j\neq i}(a_j-\varepsilon_j\lambda).
$$
We calculate:
\begin{align*}
&\sign\mathcal{R}(\varepsilon_i a_i)=\varepsilon_i(-1)^{k+i},
\\
&\sign\mathcal{R}(-\infty)=(-1)^l\sign\langle v,v\rangle_{k,l},
\\
&\sign\mathcal{R}(+\infty)=(-1)^{k-1}\sign\langle v,v\rangle_{k,l}.
\end{align*}
From there, we see that polynomial $\mathcal{R}$, which is of degree $d-1$ for space-like or time-like $\ell$ and of degree $d-2$ for light-like $\ell$, is changing sign at least $d-1$ times along the real axis, if $\langle v,v\rangle_{k,l}\neq0$, and at least $d-2$ times otherwise, so all its roots are real.
Moreover, there is one root in each of the $d-2$ intervals:
$$
(\varepsilon_i a_{i+1},\varepsilon_i a_i),\quad  i\in\{1,\dots,k-1,k+1,\dots,d-1\},
$$
and one more in $(-\infty,-a_d)$ or $(a_1,+\infty)$ if $\ell$ is space-like or time-like respectively.

Denote roots of $\mathcal{R}$ by $\zeta_0$, $\zeta_1$, \dots, $\zeta_{d-2}$, and order them in the following way:
\begin{align*}
&\zeta_i\in(a_{i+1},a_i), \ \text{for}\ 1\le i\le k-1,
  \\
&\zeta_j\in(-a_{j+2},-a_{j+1}), \ \text{for}\ k\le j\le d-2,
  \\
&\zeta_0\in(-\infty,-a_d) \ \text{for space-like}\ \ell,
  \\
&\zeta_0\in(a_1,+\infty) \ \text{for time-like}\ \ell,
  \\
&\zeta_0=\infty\ \text{for light-like}\ \ell.
\end{align*}

Note only that the right-hand side of (\ref{eq:diskriminanta}) is positive for $\zeta_0$, \dots, $\zeta_{d-2}$, it will be also be positive for $\lambda=0$, because $(A_{0}x\circ v)^2\ge0$, $A_{0}v{\circ}v>0$, and, since $x$ is inside $\mathcal{E}=\mathcal{Q}_0$, $A_{0}x{\circ}x<1$.
Notice that (\ref{eq:diskriminanta}) and the equivalent expression (\ref{eq:polinom.P}) changes sign at points $\varepsilon_i a_i$ and roots of $\mathcal{P}$ only. 
Thus, these expressions have positive values at the endpoints of each of the $d-2$ intervals:
$(\zeta_{d-2},\zeta_{d-3})$, \dots, $(\zeta_{k+1},\zeta_k)$, $(\zeta_k,0)$,
$(0,\zeta_{k-1})$, $(\zeta_{k-1},\zeta_{k-2})$, \dots, $(\zeta_2,\zeta_1)$,
and, in addition, in one of $(\zeta_0,\zeta_{d-2})$ or $(\zeta_1,\zeta_0)$,
depending if $\ell$ is space-like or time-like respectively.
Each of these $d-1$ intervals contains one point from $\{\varepsilon_i a_i\}$, thus each of them needs to contain at least one more point where expression (\ref{eq:polinom.P}) changes its sign, that is a root of $\mathcal{P}$.
Conclude that all roots of $\mathcal{P}$ are real, and that they are distributed exactly as stated in this proposition.

Now, let us consider quadrics from the confocal family containing point $x$.
Their parameters are solutions of the equation $A_{\lambda}x\circ x=1$.
Observe that $A_{\lambda}x\circ x-1$ is strictly monotonous and changes sign inside each of the following intervals:
$$
(-a_d, -a_{d-1}),\ \dots,\ (-a_{k+2},-a_{k+1}),\ (a_k,a_{k-1}),\ \dots,\ (a_2,a_1),
$$
thus it has one root in each of them.
On the other hand, for such solutions, the right-hand side of (\ref{eq:diskriminanta}) is positive,
thus there is one solution in each of the following:
$$
(c_{2l-1},c_{2l-2}),\ \dots,\ (c_3,c_2),\ (b_2,b_3),\ \dots,\ (b_{2k-2},b_{2k-1}),
$$
which makes $d-2$ solutions.
Two more solutions are placed in $(c_1,b_1)$, because $A_{0}x\circ x-1<0$ and $A_{\lambda}x\circ x-1>0$ close to the endpoints of this interval, which concludes the proof.
\end{proof}

The analog of Theorem \ref{th:parametri.kaustike} for the Euclidean space, is proved in \cite{Audin1994}.

\begin{corollary}
For each point placed inside an ellipsoid in the pseudo-Euclidean space, there are exactly two other ellipsoids from the confocal family containing this point.
\end{corollary}

%\subsection{Ellipsoidal billiards}
%\label{sec:billiards}
%\input{2-billiards/pseudobilliards}

\section{Elliptical billiard in the Minkowski plane}
\label{sec:plane}
In this part of the paper, we study  properties of confocal families of conics in the Minkowski plane, see Section \ref{sec:confocal.conics}.
We derive focal properties of such families and the corresponding elliptical billiards.
Next, in Section \ref{sec:light-like}, we study light-like trajectories of such billiards and derive a periodicity criterion in a simple form, see Theorem \ref{th:arctan}.
%The equivalence of this simple criterion with the more complicated condition of Cayley’s type is illustrated in several examples.
It is also proved in Proposition \ref{prop:rectangle} that the flow of light-like elliptical billiard trajectories is equivalent to a certain rectangular billiard flow.
%In Section \ref{sec:topology}, we conclude by a complete description of topological properties of elliptical billiards in the Minkowski plane using Fomenko invariants, see Theorem \ref{th:fomenko.minkowski}.

\subsection{Confocal conics in Minkowski plane}
\label{sec:confocal.conics}
Here, we give a review of basic properties of families of confocal
conics in the Minkowski plane.

Denote by
\begin{equation}\label{eq:ellipse}
\mathcal{E}\ :\ \frac{x^2}{a}+\frac{y^2}{b}=1
\end{equation}
an ellipse in the plane, with $a$, $b$ being fixed positive numbers.

The associated family of confocal conics is:
\begin{equation}\label{eq:confocal.conics} 
\mathcal C_{\lambda}\ :\
\frac{x^2}{a-\lambda}+\frac{y^2}{b+\lambda}=1, \quad
\lambda\in\mathbf{R}.
\end{equation}

The family is shown on Figure \ref{fig:confocal.conics}.
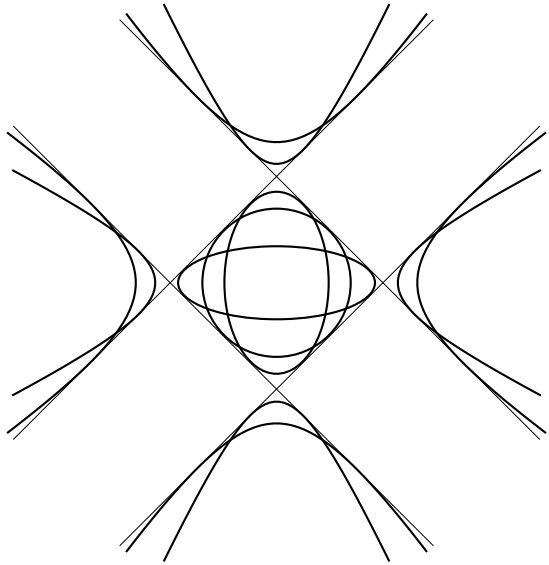
\begin{figure}[h]
% crtanje nekoliko konika iz familije $x^2/(1-\lambda)+y^2/(1+\lambda)=1$

\begin{pspicture}(-4,-4)(4,4)
  \SpecialCoor

  \psellipse(! 1 2 div sqrt 3 2 div sqrt) % elipsa lambda=1
  \psellipse(1, 1) % elipsa lambda=0
  \psellipse(! 7 4 div sqrt 1 4 div sqrt) % elipsa lambda=-3/4

  \psline[linewidth=0.35pt](! 3.5 2 sqrt 3.5 sub)(! 2 sqrt 3.5 sub 3.5)
  \psline[linewidth=0.35pt](! -3.5 3.5 2 sqrt sub)(! 3.5 2 sqrt sub -3.5)
  \psline[linewidth=0.35pt](! 3.5 3.5 2 sqrt sub)(! -3.5 2 sqrt add -3.5)
  \psline[linewidth=0.35pt](! 3.5 2 sqrt sub 3.5)(! -3.5 2 sqrt -3.5 add)

  \psplot{-2}{2}{1 x x mul 1 2.5 sub div sub 1 2.5 add mul sqrt}
  \psplot{-2}{2}{1 x x mul 1 2.5 sub div sub 1 2.5 add mul sqrt neg}

  \psplot{-1.5}{1.5}{1 x x mul 1 1.5 sub div sub 1 1.5 add mul sqrt}
  \psplot{-1.5}{1.5}{1 x x mul 1 1.5 sub div sub 1 1.5 add mul sqrt neg}

  \parametricplot{-2}{2}
   {
    1 t t mul 1 2.5 sub div sub 1 2.5 add mul sqrt
    t
   }

   \parametricplot{-2}{2}
   {
    1 t t mul 1 2.5 sub div sub 1 2.5 add mul sqrt neg
    t
   }

   \parametricplot{-1.5}{1.5}
   {
   1 t t mul 1 1.6 sub div sub 1 1.6 add mul sqrt
   t
   }

   \parametricplot{-1.5}{1.5}
   {
   1 t t mul 1 1.6 sub div sub 1 1.6 add mul sqrt neg
   t
   }

\end{pspicture}
\caption{Family of confocal conics in the Minkowski
plane.}\label{fig:confocal.conics}
\end{figure}
We may distinguish the following three subfamilies in the family
(\ref{eq:confocal.conics}):
\begin{itemize}
\item
for $\lambda\in(-b,a)$, conic $\mathcal{C}_{\lambda}$ is an ellipse;

\item
for $\lambda<-b$, conic $\mathcal{C}_{\lambda}$ is a hyperbola with $x$-axis as the major one;

\item
for $\lambda>a$, it is a hyperbola again, but now its major axis is $y$-axis.
\end{itemize}
In addition, there are three degenerated quadrics: $\mathcal{C}_{a}$, $\mathcal{C}_{b}$, $\mathcal{C}_{\infty}$ corresponding to $y$-axis, $x$-axis, and the line at the infinity respectively.
Note the following three pairs of foci:
$F_1(\sqrt{a+b},0)$, $F_2(-\sqrt{a+b},0)$; 
$G_1(0,\sqrt{a+b})$, $G_2(0,-\sqrt{a+b})$; and 
$H_1(1:-1:0)$, $H_2(1:1:0)$ on the line at the infinity.

We notice four distinguished lines: 
\begin{align*}
&x+y=\sqrt{a+b},\quad x+y=-\sqrt{a+b},\\
&x-y=\sqrt{a+b},\quad x-y=-\sqrt{a+b}.
\end{align*}
These lines
are common tangents to all conics from the confocal family.

It is elementary and straightforward to prove the following
\begin{proposition}
For each point on ellipse $\mathcal{C}_{\lambda}$, $\lambda\in(-b,a)$,
either sum or difference of its Minkowski distances from the foci $F_1$ and
$F_2$ is equal to $2\sqrt{a-\lambda}$;
either sum or difference
of the distances from the other pair of foci $G_1$, $G_2$ is equal to $2i\sqrt{b+\lambda}$.

Either sum or difference of the Minkowski distances of each point of hyperbola
$\mathcal{C}_{\lambda}$, $\lambda\in(-\infty,-b)$, from the foci
$F_1$ and $F_2$ is equal to $2\sqrt{a-\lambda}$;
for the other pair of foci $G_1$, $G_2$, it is
equal to $2\sqrt{-b-\lambda}$.

Either sum or difference of the Minkowski distances of each point of hyperbola
$\mathcal{C}_{\lambda}$, $\lambda\in(a,+\infty)$, from the foci
$F_1$ and $F_2$ is equal to $2i\sqrt{\lambda-a}$;
for the other pair of foci $G_1$, $G_2$, it is equal to $2i\sqrt{b+\lambda}$.
\end{proposition}

Billiard within an ellipse also have the famous focal property:

\begin{proposition}
Consider a billiard trajectory within ellipse $\mathcal{E}$ given by
equation (\ref{eq:ellipse}) in the Minkowski plane, such that the
line containing the initial segment of the trajectory passes through
a focus of the confocal family (\ref{eq:confocal.conics}), say
$F_1$, $G_1$, or $H_1$.
If the tangent line to $\mathcal{E}$ at the
reflection point of this segment is not light-like, then the line
containing the next segment will pass through $F_2$, $G_2$, or $H_2$
respectively.
\end{proposition}

In other words, the segments of one billiard trajectory will
alternately contain foci of one of the pairs $(F_1,F_2)$,
$(G_1,G_2)$, $(H_1,H_2)$. The only exception are successive segments
obtained by the reflection on the light-like tangent. Such segments
coincide.

\subsection{Light-like trajectories of the elliptical billiard}
\label{sec:light-like}
In this section, we are going to study light-like trajectories of
elliptical billiard in the Minkowski plane. An example of such a
billiard trajectory is shown on Figure \ref{fig:traj}.

\begin{figure}[h]
\begin{pspicture}(-3.5,-2)(3.5,2)
  \SpecialCoor
  \psellipse[unit=2cm,linewidth=0.8pt](! 2 sqrt 1)

    \psline[unit=2cm,linewidth=0.6pt](-1.40006,0.14112) (-0.654847,0.886334) (0.963496,-0.732009) (1.29718,-0.398328)
            (-0.098711,0.997561) (-1.36299,-0.266713) (-0.809946,-0.819752) (0.823021,0.813215)
            (1.35863,0.277609) (0.0827298,-0.998287) (-1.30347,0.387916) (-0.951712,0.739677)
            (0.668999,-0.881034) (1.39771,-0.152321) (0.262809,0.982581)
            (-1.22251,-0.502733)
\end{pspicture}
\caption{Light-like billiard trajectory.}\label{fig:traj}
\end{figure}
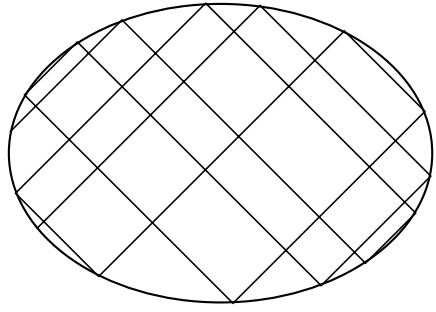

Successive segments of such trajectories are orthogonal to each
other. Notice that this implies that a trajectory can close only
after even number of reflections.

\subsubsection*{Periodic light-like trajectories}
The analytic condition for $n$-periodicity of light-like billiard trajectory within the ellipse $\mathcal{E}$ given by equation (\ref{eq:ellipse}) can be written down applying the more general Cayley's condition for closedness of a polygonal line inscribed in one conic and circumscribed about another one \cites{Cayley1853,Cayley1854}, see also \cites{LebCONIQUES,GrifHar1978}:
\begin{equation}\label{eq:cayley_plane}
\det\left(
\begin{array}{llll}
B_{3} & B_{4} & \dots & B_{m+1}\\
B_{4} & B_{5} & \dots & B_{m+2}\\
\dots & \dots & \dots & \dots\\
B_{m+1} & B_{m+2} & \dots & B_{2m-1}
\end{array}
\right) = 0,\quad\text{with}\ n=2m.
\end{equation}
Here,
$$
\sqrt{(a-\lambda)(b+\lambda)}
 =
B_0 + B_1\lambda + B_2\lambda^2 + \dots
$$
is the Taylor expansion around $\lambda=0$.

Now, we are going to derive analytic condition for periodic light-like trajectories in another way, which will lead to a more compact form of (\ref{eq:cayley_plane}).

\begin{theorem}\label{th:arctan}
Light-like billiard trajectory within ellipse $\mathcal{E}$ is
periodic with period $n$, where $n$ is an even integer if and only
if
\begin{equation}\label{eq:cayley_arctan}
\arctan\sqrt{\frac ab}\in\left\{\ \frac{k\pi}n\ \left|\ 1\le k<\frac
n2,\ \left(k,\frac{n}{2}\right)=1\right.\ \right\}.
\end{equation}
\end{theorem}

\begin{proof}
Applying the following affine transformation:
$$
(x,y)\mapsto(x\sqrt{b},y\sqrt{a}),
$$
the ellipse is transformed into a circle.
The light-like lines are transformed into lines parallel to two directions, with the angle between them equal to $2\arctan\sqrt{a/b}$.
Since the dynamics on the boundary is the rotation by this angle, the proof is complete.
\end{proof}

As an immediate consequence, we get
\begin{corollary}\label{cor:cayley.euler}
For a given even integer $n$, the number of different ratios of the
axes of ellipses having $n$-periodic light-like billiard
trajectories is equal to:
$$
\begin{cases}
\varphi(n)/2 & \text{if}\ \ n\ \text{is not divisible by}\ 4,
 \\
\varphi(n)/4 & \text{if}\ \ n\ \text{is divisible by}\ 4.
\end{cases}
$$
$\varphi$ is the Euler's totient function, i.e.\ the number of
positive integers not exceeding $n$ that are relatively prime to
$n$.
\end{corollary}

\begin{remark}\label{rem:nk}
There are four points on $\mathcal{E}$ where the tangents are
light-like. Those points cut four arcs on $\mathcal{E}$. An
$n$-periodic trajectory within $\mathcal{E}$ hits each one of a pair
of opposite arcs exactly $k$ times, and $\dfrac{n}2-k$ times the
arcs from the other pair.
\end{remark}

\subsubsection*{Light-like trajectories in ellipses and rectangular billiards}

\begin{proposition}\label{prop:rectangle}
The flow of light-like billiard trajectories within ellipse
$\mathcal{E}$ is trajectorially equivalent to the flow of those
billiard trajectories within a rectangle whose angle with the sides
is $\dfrac{\pi}4$. The ratio of the sides of the rectangle is equal
to:
$$
\frac{\pi}{2\arctan\sqrt{\dfrac{a}{b}}}-1.
$$
\end{proposition}

\begin{proof}
For the ellipses with periodic light-like trajectories, the theorem
follows from Theorem \ref{th:arctan} and Remark \ref{rem:nk}.

In other cases, the number
 $\dfrac{\pi}{2\arctan\sqrt{\dfrac{a}{b}}}-1$
is not rational. For them, the statement holds because of the
density of rational numbers.
\end{proof}

\begin{remark}
The flow of light-light billiard trajectories within a given oval in the Minkowski plane will be trajectorially equivalent to the flow of certain trajectories within a rectangle whenever invariant measure $m$ on the oval exists such that $m(AB)=m(CD)$ and $m(BC)=m(AD)$, where $A$, $B$, $C$, $D$ are points on the oval where the tangents are light-like.
\end{remark}

%\subsection{Topological properties of elliptical billiards in the Minkowski plane}
%\label{sec:topology}
%\input{3-plane/topology}

\section{Relativistic quadrics}
\label{sec:relativistic}
In this section, we are going to introduce \emph{relativistic quadrics}, as the main new object 
of the present paper, which is going to become a main tool in our further study of billiard 
dynamics.
The point is that geometric quadrics of a confocal pencil and their types in pseudo-Euclidean spaces do not satisfy usual properties of confocal quadrics in Euclidean spaces, including those necessary for applications in billiard dynamics.
For example, we have already mentioned, in the $d$-dimensional Euclidean space, there are $d$ geometric types of quadrics, while in $d$-dimensional pseudo-Euclidean space, there are $d+1$ geometric types of quadrics.
Thus, we first select those important properties of confocal families in the Euclidean spaces and axiomatize them as E1--E5 in Section \ref{sec:euclid}.
Then, in Section \ref{sec:rel2} we consider the two-dimensional case, the Minkowski plane, and we study appropriate  relativistic conics, where \cite{BirkM1962} may be seen as a historic origin of ideas of relativistic conics.
In Section \ref{sec:geom3} we study  geometrical types of quadrics in a confocal family in the three-dimensional Minkowski space.
Next, in Section \ref{sec:tropic3}, we analyze tropic curves on quadrics in the three-dimensional case and we introduce an important notion of discriminant sets $\Sigma^{\pm}$ corresponding to a confocal family.
The main facts about discriminant sets we prove in Propositions \ref{prop:tropic.surface}, \ref{prop:tropic.light}, \ref{prop:tetra.hyp}, \ref{prop:hyp.y.presek}.
Then, we study curved tetrahedra $\mathcal{T}^{\pm}$, which represent singularity sets of $\Sigma^{\pm}$ and we collect related results in Proposition \ref{prop:tetra}.
As the next important step, we introduce decorated Jacobi coordinates in Section \ref{sec:rel3} for three-dimensional Minkowski space, and we give a detailed description of the colouring into three colours. 
Each colour corresponds to a relativistic type, and we describe decomposition of a geometric quadric of each of the four  geometric types into relativistic quadrics. 
This appears to be a rather involved combinatorial-geometric problem, and we solve it by using previous analysis of discriminant surfaces. 
We give a complete description of all three relativistic types of quadrics.
In Section \ref{sec:reld}, we generalize definition of decorated Jacobi coordinates in arbitrary dimensions, and, finally,
in Proposition \ref{prop:svojstva12} we prove properties PE1 and PE2, the pseudo-Euclidean analogues of E1 and E2.

\subsection{Confocal quadrics and their types in the Euclidean space}
\label{sec:euclid}
A general family of confocal quadrics in the $d$-dimensional Euclidean space is given by:
\begin{equation}\label{eq:conf.Euclid}
\frac{x_1^2}{b_1-\lambda}+\dots+\frac{x_d^2}{b_d-\lambda}=1,\quad\lambda\in\mathbf{R}
\end{equation}
with $b_1>b_2>\dots>b_d>0$.

Such a family has the following properties:
\begin{itemize}
\item[E1]
each point of the space $\mathbf{E}^d$ is the intersection of exactly $d$ quadrics from (\ref{eq:conf.Euclid});
moreover, all these quadrics are of different geometrical types;

\item[E2]
family (\ref{eq:conf.Euclid}) contains exactly $d$ geometrical types of non-degenerate quadrics 
-- each type corresponds to one of the disjoint intervals of the parameter $\lambda$:
$(-\infty,b_d)$, $(b_d,b_{d-1})$, \dots, $(b_2,b_1)$.
\end{itemize}

The parameters $(\lambda_1,\dots,\lambda_d)$ corresponding to the quadrics of (\ref{eq:conf.Euclid}) that contain a given point in $\mathbf{E}^d$ are called \emph{Jacobi coordinates}.
We order them $\lambda_1>\dots>\lambda_d$.

Now, let us consider the motion of a billiard ball within an ellipsoid,
denote it by $\mathcal{E}$,
of the family (\ref{eq:conf.Euclid}).
Without losing generality, take that the parameter $\lambda$ corresponding to this ellipsoid be equal to $0$.
Recall that, by Chasles' theorem, each line in $\mathbf{E}^d$ is touching some $d-1$ quadrics from (\ref{eq:conf.Euclid}).
Moreover, for a line and its billiard reflection on a quadric from (\ref{eq:conf.Euclid}), the $d-1$ quadrics are the same.
This means that each segment of a given trajectory within $\mathcal{E}$ has the same $d-1$ \emph{caustics} --
denote their parameters by $\beta_1$, \dots, $\beta_{d-1}$, and introduce the following:
$$
\{\bar{b}_1,\dots,\bar{b}_{2d}\}=\{b_1,\dots,b_d,0,\beta_1,\dots,\beta_{d-1}\},
$$
such that $\bar{b}_1\ge\bar{b}_2\ge\dots\ge\bar{b}_{2d}$.
In this way, we will have $0=\bar{b}_{2d}<\bar{b}_{2d-1}$, $b_1=\bar{b}_1>\bar{b}_1$.
Moreover, it is always: $\beta_i\in\{\bar{b}_{2i},\bar{b}_{2i+1}\}$, for each $i\in\{1,\dots,d\}$, see \cite{Audin1994}.

Now, we can summarize the main properties of the flow of the Jacobi coordinates along the billiard trajectories:
\begin{itemize}
 \item[E3]
along a fixed billiard trajectory, the Jacobi coordinate $\lambda_i$ ($1\le i\le d$) takes values in segment 
$[\bar{b}_{2i-1},\bar{b}_{2i}]$;

 \item[E4]
along a trajectory, each $\lambda_i$ achieves local minima and maxima exactly at touching points with corresponding caustics, intersection points with corresponding coordinate hyper-planes, and, for $i=d$, at reflection points;

 \item[E5]
values of $\lambda_i$ at those points are $\bar{b}_{2i-1}$, $\bar{b}_{2i}$;
between the critical points, $\lambda_i$ is changed monotonously.
\end{itemize}

Those properties represent the key in the algebro-geometrical analysis of the billiard flow.

At the first glance, it seems that fine properties like those do not take place in the pseudo-Euclidean case.
In a $d$-dimensional pseudo-Euclidean space, a general confocal family contains $d+1$ geometrical types of quadrics and, in addition, quadrics of the same geometrical type have non-empty intersection.
Because of that, it looks much more complicated to analyze the billiard flow following the Jacobi-type coordinates.

In the rest of this section, we are going to overcome this important problem, by introducing a new notion of relativistic quadrics.
In our setting, we equip a geometric pencil of quadrics by an additional structure, \emph{a decoration}, which decomposes geometric quadrics of the pencil into coloured subsets which form new types of relativistic quadrics.
The new notion of relativistic quadrics, i.e.\ \emph{coloured geometric quadrics}, is more suitable for the pseudo-Euclidean geometry.
In return, we will obtain a possibility to introduce a new system of coordinates, nontrivial pseudo-Euclidean analogue of Jacobi elliptic coordinates, which is going to play a fundamental role in the sequel, as a powerfull tool in the study of separable systems.

\subsection{Confocal conics in the Minkowski plane}
\label{sec:rel2}
Let us consider first the case of the $2$-dimensional pseudo-Euclidean space $\mathbf{E}^{1,1}$, namely the Minkowski plane.
We mentioned in Section \ref{sec:confocal.conics} that a family of confocal conics in the Minkowski plane contains three geometrical types of conics: ellipses, hyperbolas with $x$-axis as the major one, and hyperbolas with $y$-axis as the major one, as shown on Figure \ref{fig:confocal.conics}.
However, it is more natural to consider \emph{relativistic conics}, which are analysed in \cite{BirkM1962}.
In this section, we give a brief account of that analysis.

Consider points $F_1(\sqrt{a+b},0)$ and $F_2(-\sqrt{a+b},0)$ in the plane.

For a given constant $c\in\mathbf{R}^{+}\cup i\mathbf{R}^{+}$, \emph{a relativistic ellipse} is the set of points $X$ satisfying:
$$
\dist_{1,1}(F_1,X)+\dist_{1,1}(F_2,X)=2c,
$$
while \emph{a relativistic hyperbola} is the union of the sets given by the following equations:
\begin{gather*}
\dist_{1,1}(F_1,X)-\dist_{1,1}(F_2,X)=2c,\\ 
\dist_{1,1}(F_2,X)-\dist_{1,1}(F_1,X)=2c.
\end{gather*}

Relativistic conics can be described as follows.
\begin{description}
 \item[$0<c<\sqrt{a+b}$]
The corresponding relativistic conics lie on ellipse $\mathcal{C}_{a-c^2}$ from family (\ref{eq:confocal.conics}).
The ellipse $\mathcal{C}_{a-c^2}$ is split into four arcs by touching points with the four common tangent lines; thus, the relativistic ellipse is the union of the two arcs intersecting the $y$-axis, while the relativistic hyperbola is the union of the other two arcs.

\item[$c>\sqrt{a+b}$]
The relativistic conics lie on $\mathcal{C}_{a-c^2}$ -- a hyperbola with $x$-axis as the major one.
Each branch of the hyperbola is split into three arcs by touching points with the common tangents; thus, the relativistic ellipse is the union of the two finite arcs, while the relativistic hyperbola is the union of the four infinite ones.

\item[$c$ is imaginary]
The relativistic conics lie on hyperbola $\mathcal{C}_{a-c^2}$ -- a hyperbola with $y$-axis as the major one.
As in the previous case, the branches are split into six arcs in total by common points with the four tangents.
The relativistic ellipse is the union of the four infinite arcs, while the relativistic hyperbola is the union of the two finite ones. 
\end{description}

The conics are shown on Figure \ref{fig:relativistic.conics}.
\begin{figure}[h]
\begin{pspicture}(-4,-4)(4,4)
  \SpecialCoor

\parametricplot{30}{150}
	{1.5 sqrt t cos mul
	 0.5 sqrt t sin mul}
\parametricplot{30}{150}
	{1.5 sqrt t cos mul
	 0.5 sqrt t sin mul neg}
\parametricplot[linestyle=dashed,dash=3pt 2pt]{-30}{30}
	{1.5 sqrt t cos mul
	 0.5 sqrt t sin mul}
\parametricplot[linestyle=dashed,dash=3pt 2pt]{-30}{30}
	{1.5 sqrt t cos mul neg
	 0.5 sqrt t sin mul}

\parametricplot{45}{135}
	{t cos
	 t sin}
\parametricplot{45}{135}
	{t cos
	 t sin neg}
\parametricplot[linestyle=dashed,dash=3pt 2pt]{-45}{45}
	{t cos
	 t sin}
\parametricplot[linestyle=dashed,dash=3pt 2pt]{-45}{45}
	{t cos neg
	 t sin}

\parametricplot{7 sqrt 1 atan}{180 7 sqrt 1 atan sub}
	{0.5 t cos mul
	 7 sqrt 2 div t sin mul}
\parametricplot{7 sqrt 1 atan}{180 7 sqrt 1 atan sub}
	{0.5 t cos mul
	 7 sqrt 2 div t sin mul neg}
\parametricplot[linestyle=dashed,dash=3pt 2pt]{7 sqrt 1 atan neg}{7 sqrt 1 atan}
	{0.5 t cos mul
	 7 sqrt 2 div t sin mul}
\parametricplot[linestyle=dashed,dash=3pt 2pt]{7 sqrt 1 atan neg}{7 sqrt 1 atan}
	{0.5 t cos mul neg
	 7 sqrt 2 div t sin mul}

  \psline[linewidth=0.35pt](! 3.5 2 sqrt 3.5 sub)(! 2 sqrt 3.5 sub 3.5)
  \psline[linewidth=0.35pt](! -3.5 3.5 2 sqrt sub)(! 3.5 2 sqrt sub -3.5)
  \psline[linewidth=0.35pt](! 3.5 3.5 2 sqrt sub)(! -3.5 2 sqrt add -3.5)
  \psline[linewidth=0.35pt](! 3.5 2 sqrt sub 3.5)(! -3.5 2 sqrt -3.5 add)

\psplot{-2}{-3 2 sqrt mul 4 div}
	{1 x x mul 1 2.5 sub div sub 1 2.5 add mul sqrt}
\psplot[linestyle=dashed,dash=3pt 2pt]{-3 2 sqrt mul 4 div}{3 2 sqrt mul 4 div}
	{1 x x mul 1 2.5 sub div sub 1 2.5 add mul sqrt}
\psplot{3 2 sqrt mul 4 div}{2}
	{1 x x mul 1 2.5 sub div sub 1 2.5 add mul sqrt}

\psplot{-2}{-3 2 sqrt mul 4 div}
	{1 x x mul 1 2.5 sub div sub 1 2.5 add mul sqrt neg}
\psplot[linestyle=dashed,dash=3pt 2pt]{-3 2 sqrt mul 4 div}{3 2 sqrt mul 4 div}
	{1 x x mul 1 2.5 sub div sub 1 2.5 add mul sqrt neg}
\psplot{3 2 sqrt mul 4 div}{2}
	{1 x x mul 1 2.5 sub div sub 1 2.5 add mul sqrt neg}

\psplot{-1.5}{2 sqrt 4 div neg}
	{1 x x mul 1 1.5 sub div sub 1 1.5 add mul sqrt}
\psplot[linestyle=dashed,dash=3pt 2pt]{2 sqrt 4 div neg}{2 sqrt 4 div}
	{1 x x mul 1 1.5 sub div sub 1 1.5 add mul sqrt}
\psplot{1.5}{2 sqrt 4 div}
	{1 x x mul 1 1.5 sub div sub 1 1.5 add mul sqrt}

\psplot{-1.5}{2 sqrt 4 div neg}
	{1 x x mul 1 1.5 sub div sub 1 1.5 add mul sqrt neg}
\psplot[linestyle=dashed,dash=3pt 2pt]{2 sqrt 4 div neg}{2 sqrt 4 div}
	{1 x x mul 1 1.5 sub div sub 1 1.5 add mul sqrt neg}
\psplot{1.5}{2 sqrt 4 div}
	{1 x x mul 1 1.5 sub div sub 1 1.5 add mul sqrt neg}

\parametricplot[linestyle=dashed,dash=3pt 2pt]{-2}{-3 2 sqrt mul 4 div}
   {
    1 t t mul 1 2.5 sub div sub 1 2.5 add mul sqrt
    t
   }
\parametricplot{-3 2 sqrt mul 4 div}{3 2 sqrt mul 4 div}
   {
    1 t t mul 1 2.5 sub div sub 1 2.5 add mul sqrt
    t
   }
\parametricplot[linestyle=dashed,dash=3pt 2pt]{2}{3 2 sqrt mul 4 div}
   {
    1 t t mul 1 2.5 sub div sub 1 2.5 add mul sqrt
    t
   }

\parametricplot[linestyle=dashed,dash=3pt 2pt]{-2}{-3 2 sqrt mul 4 div}
   {
    1 t t mul 1 2.5 sub div sub 1 2.5 add mul sqrt neg
    t
   }
\parametricplot{-3 2 sqrt mul 4 div}{3 2 sqrt mul 4 div}
   {
    1 t t mul 1 2.5 sub div sub 1 2.5 add mul sqrt neg
    t
   }
\parametricplot[linestyle=dashed,dash=3pt 2pt]{2}{3 2 sqrt mul 4 div}
   {
    1 t t mul 1 2.5 sub div sub 1 2.5 add mul sqrt neg
    t
   }

\parametricplot[linestyle=dashed,dash=3pt 2pt]{-1.5}{2 sqrt 4 div neg}
   {
   1 t t mul 1 1.6 sub div sub 1 1.6 add mul sqrt
   t
   }
\parametricplot{2 sqrt 4 div}{2 sqrt 4 div neg}
   {
   1 t t mul 1 1.6 sub div sub 1 1.6 add mul sqrt
   t
   }
\parametricplot[linestyle=dashed,dash=3pt 2pt]{1.5}{2 sqrt 4 div}
   {
   1 t t mul 1 1.6 sub div sub 1 1.6 add mul sqrt
   t
   }

\parametricplot[linestyle=dashed,dash=3pt 2pt]{-1.5}{2 sqrt 4 div neg}
   {
   1 t t mul 1 1.6 sub div sub 1 1.6 add mul sqrt neg
   t
   }
\parametricplot{2 sqrt 4 div}{2 sqrt 4 div neg}
   {
   1 t t mul 1 1.6 sub div sub 1 1.6 add mul sqrt neg
   t
   }
\parametricplot[linestyle=dashed,dash=3pt 2pt]{1.5}{2 sqrt 4 div}
   {
   1 t t mul 1 1.6 sub div sub 1 1.6 add mul sqrt neg
   t
   }

\end{pspicture}
\caption{Relativistic conics in the Minkowski plane: relativistic ellipses are represented by full lines, and hyperbolas by dashed ones.}\label{fig:relativistic.conics}
\end{figure}
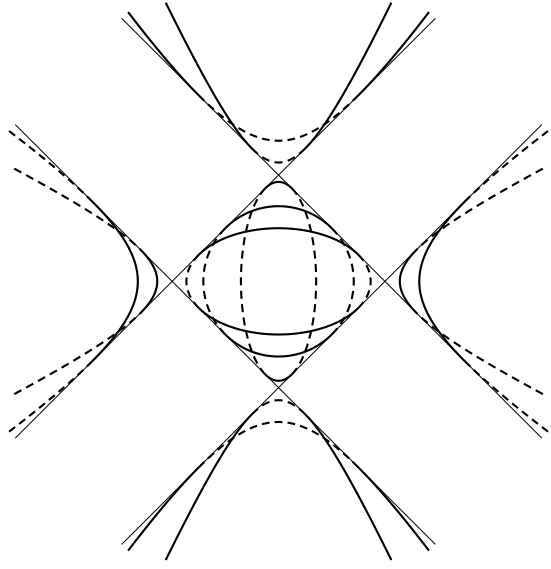

\begin{remark}
All relativistic ellipses are disjoint with each other, as well as all relativistic hyperbolas.
Moreover, at the intersection point of a relativistic ellipse which is a part of the geometric conic $\mathcal{C}_{\lambda_1}$ from the confocal family (\ref{eq:confocal.conics}) and a relativistic hyperbola belonging to $\mathcal{C}_{\lambda_2}$, it is always $\lambda_1<\lambda_2$.
\end{remark}

This remark will serve as a motivation for introducing relativistic types of quadrics in higher-dimensional pseudo-Euclidean spaces.

\subsection{Confocal quadrics in the three-dimensional Minkowski space and their geometrical types}
\label{sec:geom3}
Let us start with the three-dimensional Minkowski space $\mathbf{E}^{2,1}$.
A general confocal family of quadrics in this space is given by:
\begin{equation}\label{eq:confocal.quadrics3}
\mathcal{Q}_{\lambda}\ :\ \frac{x^2}{a-\lambda}+\frac{y^2}{b-\lambda}+\frac{z^2}{c+\lambda}=1,
\quad
\lambda\in\mathbf{R},
\end{equation}
with $a>b>0$, $c>0$.

The family (\ref{eq:confocal.quadrics3}) contains four geometrical types of quadrics:

\begin{itemize}
\item
$1$-sheeted hyperboloids oriented along $z$-axis, for $\lambda\in(-\infty,-c)$;

\item
ellipsoids, corresponding to $\lambda\in(-c,b)$;

\item
$1$-sheeted hyperboloids oriented along $y$-axis, for $\lambda\in(b,a)$;

\item
$2$-sheeted hyperboloids, for $\lambda\in(a,+\infty)$ -- these hyperboloids are oriented along $z$-axis.
\end{itemize}
In addition, there are four degenerated quadrics: $\mathcal{Q}_{a}$, $\mathcal{Q}_{b}$, $\mathcal{Q}_{-c}$, $\mathcal{Q}_{\infty}$, that is planes $x=0$, $y=0$, $z=0$, and the plane at the infinity respectively.
In the coordinate planes, we single out the following conics: 
\begin{itemize}
\item
hyperbola $\mathcal{C}^{yz}_{a}\ :\ -\dfrac{y^2}{a-b}+\dfrac{z^2}{c+a}=1$ in the plane $x=0$;

\item
ellipse $\mathcal{C}^{xz}_{b}\ :\ \dfrac{x^2}{a-b}+\dfrac{z^2}{c+b}=1$ in the plane $y=0$;

\item
ellipse $\mathcal{C}^{xy}_{-c}\ :\ \dfrac{x^2}{a+c}+\dfrac{y^2}{b+c}=1$ in the plane $z=0$.
\end{itemize}

\subsection{Tropic curves on quadrics in the three-dimensional Minkowski space and discriminant sets $\Sigma^{\pm}$}
\label{sec:tropic3}
On each quadric, notice the \emph{tropic curves} -- the set of points where the induced metrics on the tangent plane is degenerate.

Since the tangent plane at point $(x_0,y_0,z_0)$ of $\mathcal{Q}_{\lambda}$ is given by the equation:
$$
\frac{xx_0}{a-\lambda}+\frac{yy_0}{b-\lambda}+\frac{zz_0}{c+\lambda}=1,
$$
and the induced metric is degenerate if and only if the parallel plane that contains the origin is tangent to the light-like cone $x^2+y^2-z^2=0$, i.e.:
$$
\frac{x_0^2}{(a-\lambda)^2}+\frac{y_0^2}{(b-\lambda)^2}-\frac{z_0^2}{(c+\lambda)^2}=0,
$$
we come to the statement formulated in \cite{KhTab2009}:

\begin{proposition}\label{prop:cone.tropic}
The tropic curves on $\mathcal{Q}_{\lambda}$ is the intersection of the quadric with the cone:
\begin{equation*}
\frac{x^2}{(a-\lambda)^2}+\frac{y^2}{(b-\lambda)^2}-\frac{z^2}{(c+\lambda)^2}=0.
\end{equation*}
\end{proposition}

Now, consider the set of the tropic curves on all quadrics of the family (\ref{eq:confocal.quadrics3}).
From Proposition \ref{prop:cone.tropic}, we get:

\begin{proposition}\label{prop:tropic.surface}
The union of the tropic curves on all quadrics of (\ref{eq:confocal.quadrics3}) is a union of two ruled surfaces $\Sigma^+$ and $\Sigma^-$ which can be parametrically represented as:
\begin{align*}
&\Sigma^+\ :\quad
x = \frac{a-\lambda}{\sqrt{a+c}}\cos t,\quad
y = \frac{b-\lambda}{\sqrt{b+c}}\sin t,\quad
z = (c+\lambda)\sqrt{\frac{\cos^2t}{a+c}+\frac{\sin^2t}{b+c}},\\
&\Sigma^-\ :\quad
x = \frac{a-\lambda}{\sqrt{a+c}}\cos t,\quad
y = \frac{b-\lambda}{\sqrt{b+c}}\sin t,\quad
z = -(c+\lambda)\sqrt{\frac{\cos^2t}{a+c}+\frac{\sin^2t}{b+c}},\\
&\text{with}\ \lambda\in\mathbf{R},\quad t\in[0,2\pi).
\end{align*}
The intersection of these two surfaces is an ellipse in the $xy$-plane:
$$
\Sigma^+\cap\Sigma^-\ : \ \frac{x^2}{a+c}+\frac{y^2}{b+c}=1,\ z=0.
$$
The two surfaces $\Sigma^+$, $\Sigma^-$ are developable as embedded into Euclidean space.
Moreover, their generatrices are all light-like.
\end{proposition}
\begin{proof}
Denote by $\mathbf{r}=(x,y,z)$ an arbitrary point of $\Sigma^{+}\cup\Sigma^{-}$, and by $\mathbf{n}$ the corresponding unit normal vector, 
$\mathbf{n}={\mathbf{r}_{\lambda}\times\mathbf{r}_t}/{|\mathbf{r}_{\lambda}\times\mathbf{r}_t|}$.
Here, by $\times$ we denoted the vector product in the three-dimensional Euclidean space.
Then, the Gaussian curvature of the surface is $K=(LN-M^2)/(EG-F^2)$, with
$L=\mathbf{r}_{\lambda\lambda}\cdot\mathbf{n}=0$,
$M=\mathbf{r}_{\lambda t}\cdot\mathbf{n}=0$,
$N=\mathbf{r}_{tt}\cdot\mathbf{n}$,
$E=\mathbf{r}_{\lambda}\cdot\mathbf{r}_{\lambda}$,
$F=\mathbf{r}_{\lambda}\cdot\mathbf{r}_{t}$,
$G=\mathbf{r}_{t}\cdot\mathbf{r}_{t}$.
Since
$$
EG-F^2=\frac{\left(a+b-2\lambda+(b-a)\cos(2t)\right)^2}{2(a+c)(b+c)}\not\equiv0,
$$
the Gaussian curvature $K$ is equal to zero.
\end{proof}
Surfaces $\Sigma^+$ and $\Sigma^-$ from Proposition \ref{prop:tropic.surface} are represented on Figure \ref{fig:tropic.surface}.
\begin{figure}[h]
\includegraphics[width=5.6cm, height=7.5cm]{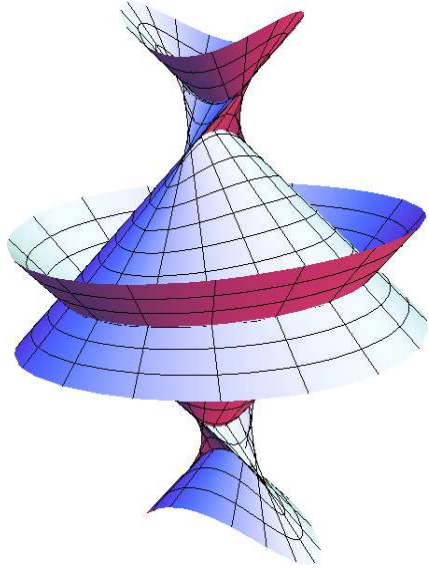}
\caption{The union of all tropic curves of a confocal family.}\label{fig:tropic.surface}
\end{figure}

In \cite{Pei1999}, a definition of generalization of Gauss map to surfaces in the three-dimensional Minkowski space is suggested.
Namely, the \emph{pseudo vector product} introduced as:
$$
\mathbf{x}\wedge\mathbf{y}=
\left|\begin{array}{rrr}
x_1 & x_2 & x_3\\
y_1 & y_2 & y_3\\
e_1 & e_2 & -e_3
\end{array}\right|=(x_2y_3-x_3y_2,\ x_3y_1-x_1y_3,\ -(x_1y_2-x_2y_1))=E_{2,1}(\mathbf{x}\times\mathbf{y}).
$$
It is easy to check that
$\left<\mathbf{x}\wedge\mathbf{y},\mathbf{x}\right>_{2,1}=\left<\mathbf{x}\wedge\mathbf{y},\mathbf{y}\right>_{2,1}=0$.

Then, for surface $S:U\to\mathbf{E}^{2,1}$, with $U\subset\mathbf{R}^2$,
\emph{the Minkowski Gauss map} is defined as:
$$
\mathcal{G}\ :\ U\to\mathbf{RP}^2,
\quad
\mathcal{G}(x_1,x_2)= \mathbf{P}\left(\frac{\partial S}{\partial x_1}\wedge\frac{\partial S}{\partial x_2}\right),
$$
where $\mathbf{P}:\mathbf{R}^3\setminus\{(0,0,0)\}\to\mathbf{RP}^2$ is the usual projectivization.

\begin{lemma}\label{lema:tropic.s.ll}
The Minkowski Gauss map of surfaces $\Sigma^{\pm}$ is singular at all points.
\end{lemma}
\begin{proof}
This follows from the fact that $\mathbf{r}_{\lambda}\wedge\mathbf{r}_t$ is light-like for all $\lambda$ and $t$.
\end{proof}

%From Lemma \ref{lema:tropic.s.ll}, it follows by \cite{CI2010} that surfaces $\Sigma^{\pm}$ are developable as embedded in the Euclidean space, which we already proved directly in Proposition \ref{prop:tropic.surface}.
%Moreover, 
Since \emph{the pseudo-normal vectors} to $\Sigma^{\pm}$ are all light-like, these surfaces are \emph{light-like developable}, as defined in \cite{CI2010}.
There, a classification of such surfaces is given -- each is one part-by-part contained in the following:
\begin{itemize}
 \item
a light-like plane;
 \item
a light-like cone;
 \item
a tangent surface of a light-like curve.
\end{itemize}
Since $\Sigma^{+}$ and $\Sigma^{-}$ are contained neither in a plane nor in a cone, we expect that they will be tangent surfaces of some light-like curve, which is going to be shown in the sequel, see Corollary \ref{cor:tangent.surface} later in this section. 

On each of the surfaces $\Sigma^{+}$, $\Sigma^{-}$, we can notice that tropic lines corresponding to $1$-sheeted hyperboloids oriented along $y$-axies form one curved tetrahedron, see Figure \ref{fig:tropic.surface}.
Denote the tetrahedra by $\mathcal{T}^{+}$ and $\mathcal{T}^{-}$ respectively: they are symmetric with respect to the $xy$-plane.
On Figure \ref{fig:tetra}, tetrahedron $\mathcal{T}^{+}\subset\Sigma^{+}$ is shown.
\begin{figure}[h]
\includegraphics[width=6cm, height=6cm]{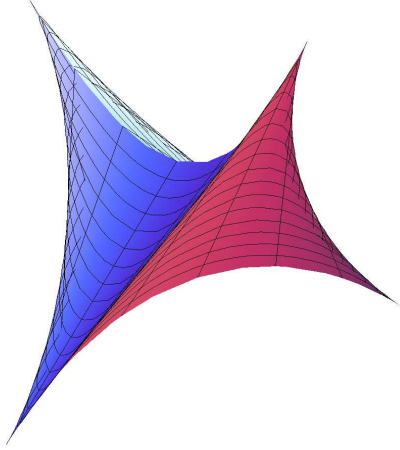}
\caption{Curved thetrahedron $\mathcal{T}^{+}$: the union of all tropic curves on $\Sigma^{+}$ corresponding to $\lambda\in(b,a)$.}\label{fig:tetra}
\end{figure}

Let us summarize the properties of these tetrahedra.
\begin{proposition}\label{prop:tetra}
Consider the subset $\mathcal{T}^{+}$ of $\Sigma^{+}$ determined by the condition $\lambda\in[b,a]$.
This set is a curved tetrahedron, with the following properties:
\begin{itemize}
 \item 
its verteces are:
\begin{align*}
&V_1\left(\frac{a-b}{\sqrt{a+c}},\ 0,\ \frac{b+c}{\sqrt{a+c}}\right),\quad
V_2\left(-\frac{a-b}{\sqrt{a+c}},\ 0,\ \frac{b+c}{\sqrt{a+c}}\right),\\
&V_3\left(0,\ \frac{a-b}{\sqrt{b+c}},\ \frac{a+c}{\sqrt{b+c}}\right),\quad
V_4\left(0,\ -\frac{a-b}{\sqrt{b+c}},\ \frac{a+c}{\sqrt{b+c}}\right);
\end{align*}
 \item
the shorter arcs of conics $\mathcal{C}^{xz}_{b}$ and $\mathcal{C}^{yz}_{a}$ determined by $V_1$, $V_2$ and $V_3$, $V_4$ respectively are two edges of the tetrahedron;
 \item
those two edges represent self-intersection of $\Sigma^{+}$;
 \item
other four edges are determined by the relation:
\begin{equation}\label{eq:cusp}
-a-b+2\lambda+(a-b)\cos2t=0,
\end{equation}
 \item
those four edges are cuspidal edges of $\Sigma^{+}$;
 \item
thus, at each vertex of the tetrahedron, a swallowtail singularity of $\Sigma^+$ occurs.
\end{itemize}
\end{proposition}

\begin{proof}
Equation (\ref{eq:cusp}) is obtained from the condition $\mathbf{r}_{t}\times\mathbf{r}_{\lambda}=0$.
\end{proof}

\begin{lemma}\label{lemma:tropic.double}
The tropic curves of the quadric $\mathcal{Q}_{\lambda_0}$ represent exactly the locus of points $(x,y,z)$ where equation
\begin{equation}\label{eq:jednacina}
\frac{x^2}{a-\lambda}+\frac{y^2}{b-\lambda}+\frac{z^2}{c+\lambda}=1
\end{equation}
has $\lambda_0$ as a multiple root.
\end{lemma}

\begin{proof}
Without losing generality, take $\lambda_0=0$.
Equation (\ref{eq:jednacina}) is equivalent to:
\begin{equation}\label{eq:polynomial.jacobi}
\lambda^3+q_2\lambda^2+q_1\lambda+q_0=0,
\end{equation}
with
\begin{align*}
q_2 &= -x^2-y^2+z^2+a+b-c,\\
q_1 &= x^2(b-c) + y^2(a-c) -z^2(a+b) -ab +bc+ac,\\
q_0 &= x^2bc+y^2ac+z^2ab-abc.
\end{align*}
Polynomial (\ref{eq:polynomial.jacobi}) has $\lambda_0=0$ as a double zero if and only if $q_0=q_1=0$.
Obviously, $q_0=0$ is equivalent to the condition that $(x,y,z)$ belongs to $\mathcal{Q}_{0}$.
On the other hand, we have:
\begin{align*}
q_1&=x^2(b-c) + y^2(a-c) -z^2(a+b) -ab +bc+ac\\
&=(ab-bc+ac)\left(\frac{x^2}{a}+\frac{y^2}{b}+\frac{z^2}{c}-1\right) - 
abc\left(\frac{x^2}{a^2}+\frac{y^2}{b^2}-\frac{z^2}{c^2}\right),
\end{align*}
which is needed.
\end{proof}

\begin{proposition}\label{prop:tropic.light}
A tangent line to the tropic curve of a non-degenarate quadric of the family (\ref{eq:confocal.quadrics3}) is always space-like, except on a $1$-sheeted hyperboloid oriented along $y$-axis.

Tangent lines of a tropic on $1$-sheeted hyperboloids oriented along $y$-axis are light-like exactely at four points, while at other points of the tropic curve, the tangents are space-like.

Moreover, a tangent line to the tropic of a quadric from (\ref{eq:confocal.quadrics3}) belongs to the quadric if and only if it is light-like.
\end{proposition}

\begin{figure}[h]
\includegraphics[width=6cm, height=8.4cm]{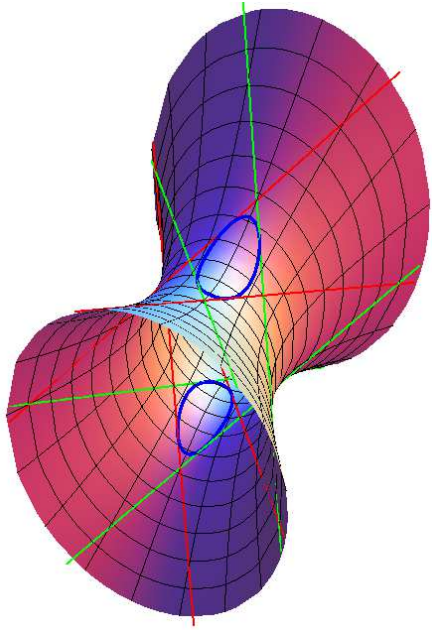}
\caption{The tropic curves and its light-like tangents on a hyperboloid.}\label{fig:hip-tropic-light}
\end{figure}

\begin{proof}
The tropic curves on $\mathcal{Q}_{\lambda}$, similarly as in Proposition \ref{prop:tropic.surface}, can be represented as:
$$
x(t) = \frac{a-\lambda}{\sqrt{a+c}}\cos t,\quad
y(t) = \frac{b-\lambda}{\sqrt{b+c}}\sin t,\quad
z(t) = \pm(c+\lambda)\sqrt{\frac{\cos^2t}{a+c}+\frac{\sin^2t}{b+c}},
$$
with $t\in[0,2\pi)$.

We calculate:
$$
\dot{x}^2+\dot{y}^2-\dot{z}^2=\frac{(a+b-2\lambda-(a-b)\cos2t)^2}{2(a+b+2c-(a-b)\cos2t)},
$$
which is always non-negative, and may attain zero only if $\lambda\in[b,a]$.
For $\lambda\in(b,a)$, there are exactly four values of $t$ in $[0,2\pi)$ where the expression attains zero.

Now, fix $t\in[0,2\pi)$ and $\lambda\in\mathbf{R}$.
The tangent line to the tropic of $\mathcal{Q}_{\lambda}$ at $(x(t),y(t),z(t))$ is completely contained in $\mathcal{Q}_{\lambda}$ if and only if, for each $\tau$:
$$
\frac{(x(t)+\tau\dot{x}(t))^2}{a-\lambda} + \frac{(y(t)+\tau\dot{y}(t))^2}{b-\lambda} + 
\frac{(z(t)+\tau\dot{z}(t))^2}{c+\lambda}=1,
$$
which is equivalent to:
$$
\frac{a+b-2\lambda-(a-b)\cos2t}{a+b+2c-(a-b)\cos2t}=0.
$$
\end{proof}

\begin{remark}
In other words,
the only quadrics of the family (\ref{eq:confocal.quadrics3}) that may contain a tangent to its tropic curve are  $1$-sheeted hyperboloids oriented along $y$-axis,
and those tangents are always light-like.
The tropic curves and their light-like tangents on such an hyperboloid are shown on Figure \ref{fig:hip-tropic-light}.
\end{remark}

Notice that equations obtained in the proof of Proposion \ref{prop:tropic.light} are equivalent to equation (\ref{eq:cusp}) of Proposition \ref{prop:tetra}, which leads to the following:

\begin{proposition}\label{prop:tetra.hyp}
Each generatrix of $\Sigma^{+}$ and $\Sigma^{-}$ is contained in one $1$-sheeted hyperboloid oriented along $y$-axis from (\ref{eq:confocal.quadrics3}).
Moreover, such a generatrix is touching at the same point one of the tropic curves of the hyperboloid and one of the cusp-like edges of the corresponding curved tetrahedron. 
\end{proposition}

\begin{corollary}\label{cor:tangent.surface}
Surfaces $\Sigma^{+}$ and $\Sigma^{-}$ are tangent surfaces of the cuspidal edges of thetrahedra $\mathcal{T}^+$ and $\mathcal{T}^-$ respectively.
\end{corollary}

In next propositions, we give further analysis the light-like tangents to the tropic curves on an $1$-sheeted hyperboloid oriented along $y$-axis.

\begin{proposition}\label{prop:hyp.y.presek}
For a fixed $\lambda_0\in(b,a)$,
consider a hyperboloid $\mathcal{Q}_{\lambda_0}$ from (\ref{eq:confocal.quadrics3}) and an arbitrary point $(x,y,z)$ on $\mathcal{Q}_{\lambda_0}$.
Equation (\ref{eq:jednacina}) has, along with $\lambda_0$, two other roots in $\mathbf{C}$: denote them by $\lambda_1$ and $\lambda_2$.
Then $\lambda_1$=$\lambda_2$ if and only if $(x,y,z)$ is placed on a light-like tangent to a tropic curve of $\mathcal{Q}_{\lambda_0}$.
\end{proposition}

\begin{proof}
Follows from the fact that the light-like tangents are contained in the $\Sigma^{+}\cup\Sigma^{-}$, see Proposition \ref{prop:tropic.surface}, Lemma \ref{lemma:tropic.double}, and Proposition \ref{prop:tetra.hyp}.
\end{proof}

\begin{proposition}
Two light-like lines on a one-sheeted hyperboloid oriented along $y$-axis from (\ref{eq:confocal.quadrics3}) are either skew or intersect each other on a degenerate quadric from (\ref{eq:confocal.quadrics3}).
\end{proposition}

\begin{proof}
Follows from the fact that the hyperboloid is symmetric with respect to the coordinate planes.
\end{proof}

\begin{lemma}\label{lemma:tropic.non.y}
Consider a non-degenerate quadric $\mathcal{Q}_{\lambda_0}$, which is not a hyperboloid oriented along $y$-axis,
i.e.\ $\lambda_0\not\in[b,a]\cup\{-c\}$.
Then each point of $\mathcal{Q}_{\lambda_0}$ which is not on one of the tropic curves is contained in two additional distinct quadrics from the family (\ref{eq:confocal.quadrics3}).

Consider two points $A$, $B$ of $\mathcal{Q}_{\lambda_0}$, which are placed in the same connected component bounded by the tropic curves, and denote by $\lambda'_A$, $\lambda''_A$ and $\lambda'_B$, $\lambda''_B$ the solutions, different than $\lambda_0$, of equation (\ref{eq:jednacina}) corresponding to $A$ and $B$ respectively.
Then, if $\lambda_0$ is smaller (resp.\ bigger, between) than $\lambda'_A$, $\lambda''_A$, it is also smaller (resp.\ bigger, between) than $\lambda'_B$, $\lambda''_B$.
\end{lemma}

\begin{lemma}\label{lemma:tropic.y}
Let $\mathcal{Q}_{\lambda_0}$ be a hyperboloid oriented along $y$-axis, $\lambda_0\in(b,a)$, and $A$, $B$ two points of $\mathcal{Q}_{\lambda_0}$, which are placed in the same connected component bounded by the tropic curves and light-like tangents. 
Then, if $A$ is contained in two more quadrics from the family (\ref{eq:confocal.quadrics3}), the same is true for $B$.

In this case, denote by $\lambda'_A$, $\lambda''_A$ and $\lambda'_B$, $\lambda''_B$ the real solutions, different than $\lambda_0$, of equation (\ref{eq:jednacina}) corresponding to $A$ and $B$ respectively.
Then, if $\lambda_0$ is smaller (resp.\ bigger, between) than $\lambda'_A$, $\lambda''_A$, it is also smaller (resp.\ bigger, between) than $\lambda'_B$, $\lambda''_B$.

On the other hand, if $A$ is not contained in any other quadric from (\ref{eq:confocal.quadrics3}), then the same is true for all points of its connected component.
\end{lemma}

\begin{proof}
The proof of both Lemmae \ref{lemma:tropic.non.y} and \ref{lemma:tropic.y} follows from the fact that the solutions of (\ref{eq:jednacina}) are continuously changed through the space and that two of the solutions coincide exactly on tropic curves and their light-like tangents.
\end{proof}

\subsection{Generalized Jacobi coordinates and relativistic quadrics in the three-dimensional Minkowski space}
\label{sec:rel3}
\begin{definition}
\emph{Generalized Jacobi coordinates} of point $(x,y,z)$ in the three-di\-men\-sio\-nal Minkowski space $\mathbf{E}^{2,1}$ is the unordered triplet of solutions of equation (\ref{eq:jednacina}).
\end{definition}

Note that any of the following cases may take place:
\begin{itemize}
 \item
generalized Jacobi coordinates are real and different;
 \item
only one generalized Jacobi coordinate is real;
 \item
generalized Jacobi coordinates are real, but two of them coincide;
 \item
all three generalized Jacobi coordinates are equal. 
\end{itemize}

Lemmae \ref{lemma:tropic.non.y} and \ref{lemma:tropic.y} will help us to define relativistic types of quadrics in the $3$-dimensional Minkowski space.
Consider connected components of quadrics from (\ref{eq:confocal.quadrics3}) bounded by tropic curves and, for $1$-sheeted hyperboloids oriented along $y$-axis, their light-light tangent lines.
Each connected component will represent \emph{a relativistic quadric}.

\begin{definition}
A component of quadric $\mathcal{Q}_{\lambda_0}$ is \emph{of relativistic type $E$} if, at each of its points, $\lambda_0$ is smaller than the other two generalized Jacobi coordinates.

A component of quadric $\mathcal{Q}_{\lambda_0}$ is \emph{of relativistic type $H^1$} if, at each of its points, $\lambda_0$ is between the other two generalized Jacobi coordinates.

A component of quadric $\mathcal{Q}_{\lambda_0}$ is \emph{of relativistic type $H^2$} if, at each of its points, $\lambda_0$ is bigger than the other two generalized Jacobi coordinates.

A component of quadric $\mathcal{Q}_{\lambda_0}$ is \emph{of relativistic type $0$} if, at each of its points, $\lambda_0$ is the only real generalized Jacobi coordinate.
\end{definition}

Lemmae \ref{lemma:tropic.non.y} and \ref{lemma:tropic.y} guarantee that types of relativistic quadrics are well-defined, i.e.\ that to each such a quadric a unique type $E$, $H^1$, $H^2$, or $0$ can be assigned.

\begin{definition}
Suppose $(x,y,z)$ is a point of the three-di\-men\-sio\-nal Minkowski space $\mathbf{E}^{2,1}$ where equation \ref{eq:jednacina} has real and different solutions.
\emph{Decorated Jacobi coordinates} of that point is the ordered triplet of pairs:
$$
(E,\lambda_1),\quad
(H^1,\lambda_2),\quad
(H^2,\lambda_3),
$$
of generalized Jacobi coordinates and the corresponding types of relativistic quadrics.
\end{definition}

Now, we are going to analyze the arrangement of the relativistic quadrics.
Let us start with their intersections with the coordinate planes.

\subsubsection*{Intersection with the $xy$-plane}

In the $xy$-plane, the Minkowski metrics is reduced to the Euclidean one.
The family (\ref{eq:confocal.quadrics3}) is intersecting this plane by the following family of confocal conics:
\begin{equation}\label{eq:conf.xy}
\mathcal{C}^{xy}_{\lambda}\ :\ \frac{x^2}{a-\lambda}+\frac{y^2}{b-\lambda}=1.
\end{equation}

We conclude that the $xy$-plane is divided by ellipse $\mathcal{C}^{xy}_{-c}$ into two relativistic quadrics:
\begin{itemize}
 \item 
the region within $\mathcal{C}^{xy}_{-c}$ is a relativistic quadric of $E$-type;
 \item
the region outside this ellipse is of $H^1$-type.
\end{itemize}

Moreover, the types of relativic quadrics intersecting the $xy$-plane are:
\begin{itemize}
 \item
the components of ellipsoids are of $H^1$-type;
 \item
the components of $1$-sheeted hyperboloids oriented along $y$-axis of $H^2$-type;
 \item
the components of $1$-sheeted hyperboloids oriented along $z$-axis of $E$-type.
\end{itemize}

\subsubsection*{Intersection with the $xz$-plane}

In the $xz$-plane, the reduced metrics is the Minkowski one.
The intersection of family (\ref{eq:confocal.quadrics3}) with this plane is the following family of confocal conics:
\begin{equation}\label{eq:conf.xz}
\mathcal{C}^{xz}_{\lambda}\ :\ \frac{x^2}{a-\lambda}+\frac{z^2}{c+\lambda}=1.
\end{equation}

The plane is divided by ellipse $\mathcal{C}^{xz}_{b}$ and the four joint tangents of (\ref{eq:conf.xz}) into $13$ parts:
\begin{itemize}
\item
the part within $\mathcal{C}^{xz}_{b}$ is a relativistic quadric of $H^1$-type;

\item
four parts placed outside of $\mathcal{C}^{xz}_{b}$ that have non-empty intersection with the $x$-axis are of $H^2$-type;

\item
four parts placed outside of $\mathcal{C}^{xz}_{b}$ that have non-empty intersection with the $z$-axis are of $E$-type;

\item
the four remaining parts are of $0$-type and no quadric from the family (\ref{eq:confocal.quadrics3}), except the degenerated $\mathcal{Q}_{b}$, is passing through any of their points.
\end{itemize}

\subsubsection*{Intersection with the $yz$-plane}

As in the previous case, in the $yz$-plane, the reduced metrics is the Minkowski one.
The intersection of family (\ref{eq:confocal.quadrics3}) with this plane is the following family of confocal conics:
\begin{equation}\label{eq:conf.yz}
\mathcal{C}^{yz}_{\lambda}\ :\ \frac{y^2}{b-\lambda}+\frac{z^2}{c+\lambda}=1.
\end{equation}

The plane is divided by hyperbola $\mathcal{C}^{yz}_{a}$ and joint tangents of (\ref{eq:conf.yz}) into $15$ parts:
\begin{itemize}
\item
the two convex parts determined by $\mathcal{C}^{yz}_{a}$ are relativistic quadric of $H^1$-type;

\item
five parts placed outside of $\mathcal{C}^{yz}_{a}$ that have non-empty intersection with the coordinate axes are of $H^2$-type;

\item
four parts, each one placed between $\mathcal{C}^{yz}_{a}$ and one of the joint tangents of (\ref{eq:conf.yz}) are of $E$-type;

\item
through points of the four remaining parts no quadric from the family (\ref{eq:confocal.quadrics3}), 
except the degenerated $\mathcal{Q}_{a}$, is passing.
\end{itemize}

Intersection of relativistic quadrics with the coordinate planes is shown in Figure \ref{fig:relativistic.3d}.
There, the type $E$ quadrics are coloured in dark gray, type $H^1$ medium gray, type $H^2$ light gray, while quadrics of type $0$ are white.
Curves $\mathcal{C}^{xy}_{-c}$, $\mathcal{C}^{xz}_{b}$, $\mathcal{C}^{yz}_{a}$ are also white in the figure.

\begin{figure}[h]
\input{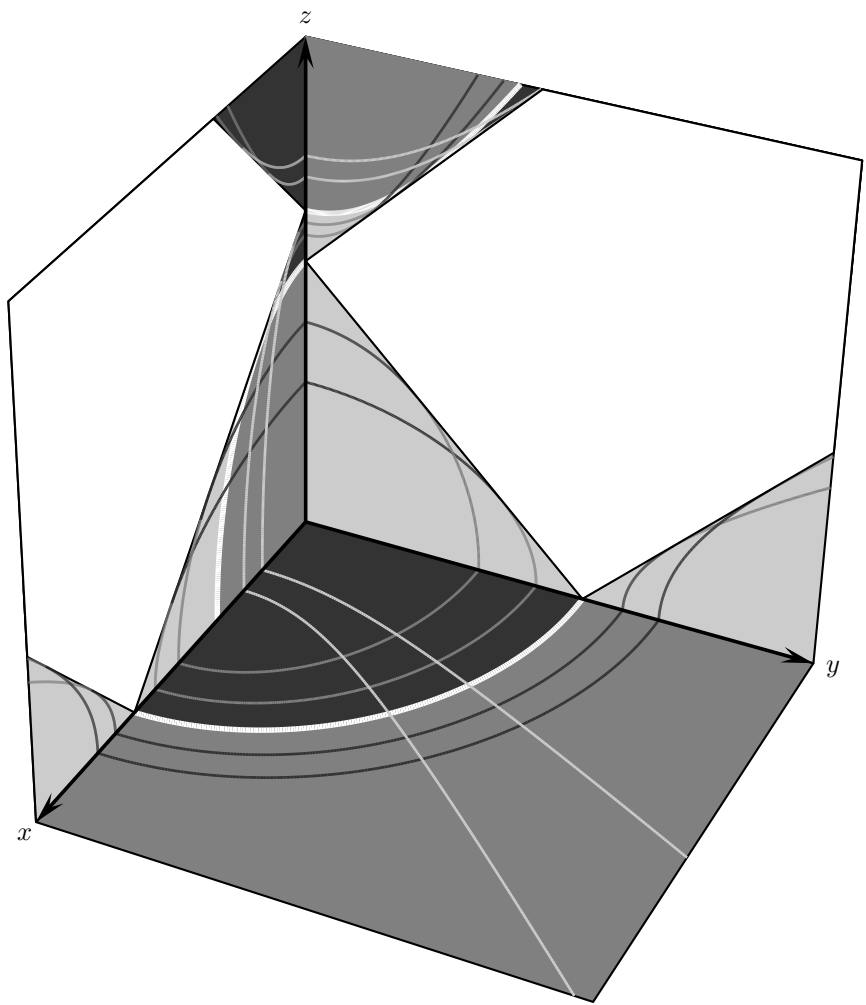}
\caption{Intersection of relativistic quadrics with coordinate planes.}\label{fig:relativistic.3d}
\end{figure}

Let us notice that from the above analysis, using Lemmae \ref{lemma:tropic.non.y} and \ref{lemma:tropic.y}, we can determine the type of each relativistic quadric with a non-empty intersection with some of the coordinate hyper-planes.

\subsubsection*{$1$-sheeted hyperboloids oriented along $z$-axis: $\lambda\in(-\infty,-c)$}
Such a hyperboloid is divided by its tropic curves into three connected components -- two of them are unbounded and mutually symmetric with respect to the $xy$-plane, while the third one is the bounded annulus placed between them.
The two symmetric ones are of $H^1$-type, while the third one is of $E$-type.

\subsubsection*{Ellipsoids: $\lambda\in(-c,b)$}
An ellipsoid is divided by the tropic curves into three bounded connected components -- two of them are mutually symmetric with respect to the $xy$-plane, while the third one is the annulus placed between them.
In this case, the symmetric components represent relativistic quadrics of $E$-type.
The annulus is of $H^1$-type.

\subsubsection*{$1$-sheeted hyperboloids oriented along $y$-axis: $\lambda\in(b,a)$}
The decomposition of those hyperboloids into relativistic quadrics is more complicated and interesting than for the other types of quadrics from (\ref{eq:confocal.quadrics3}).
By its two tropic curves and their eight light-like tangent lines, such a hyperboloid is divided into $28$ connected components:
\begin{itemize}
 \item
two bounded components placed inside the tropic curves are of $H^1$-type;

 \item
four bounded components placed between the tropic curves and light-like tangents, such that they have non-empty intersections with $xz$-plane are of $H^2$-type;

 \item
four bounded components placed between the tropic curves and light-like tangents, such that they have non-empty intersections with $yz$-plane are of $E$-type;

 \item
two bounded components, each limited by four light-like tangents, are of $H^2$-type;

 \item
four unbounded components, each limited by two light-like tangents, such that they have non-empty intersections with the $xy$-plane, are of $H^2$-type;

 \item
four unbounded components, each limited by two light-like tangents, such that they have non-empty intersections with the $yz$-plane, are of $E$-type;

 \item
eight unbounded components, each limited by four light-like tangents, are sets of points not contained in any other quadric from (\ref{eq:confocal.quadrics3}).
\end{itemize}

\subsubsection*{$2$-sheeted hyperboloids: $\lambda\in(a,+\infty)$}
Such a hyperboloid is by its tropic curves divided into four connected components:
two bounded ones are of $H^2$-type, while the two unbounded are of $H^1$-type.

\subsection{Decorated Jacobi coordinates and relativistic quadrics in $d$-di\-men\-sio\-nal pse\-udo-Euclidean space}
\label{sec:reld}
Inspired by results obtained in Sections \ref{sec:rel2} and \ref{sec:rel3},
now we are going to introduce relativistic quadrics and their types in confocal family (\ref{eq:confocal}) in the $d$-dimensional pseudo-Euclidean space $\mathbf{E}^{k,l}$.

\begin{definition}
\emph{Generalized Jacobi coordinates} of point $x$ in the $d$-di\-men\-sio\-nal pseudo-Euclidean space $\mathbf{E}^{k,l}$ is the unordered $d$-tuple of solutions $\lambda$ of equation:
\begin{equation}\label{eq:jednacinad}
\frac{x_1^2}{a_1-\lambda} +\dots+ \frac{x_k^2}{a_k-\lambda} +
\frac{x_{k+1}^2}{a_{k+1}+\lambda} +\dots+
\frac{x_d^2}{a_d+\lambda}=1.
\end{equation} 
\end{definition}

As already mentioned in Section \ref{sec:confocal}, this equation has either $d$ or $d-2$ real solutions.
Besides, some of the solutions may be multiple.

The set $\Sigma_d$ of points $x$ in $\mathbf{R}^d$ where equation (\ref{eq:jednacinad}) has multiple solutions is an algebraic hyper-suface.
$\Sigma_d$ divides each quadric from (\ref{eq:confocal}) into several connected components.
We call these components \emph{relativistic quadrics}.

Since the generalized Jacobi coordinates depend continuosly on $x$, the following definition can be made:
\begin{definition}
We say that a relativistic quadric placed on $\mathcal{Q}_{\lambda_0}$ is \emph{of type $E$} if, at each of its points, $\lambda_0$ is smaller than the other $d-1$ generalized Jacobi coordinates.

We say that a relativistic quadric placed on $\mathcal{Q}_{\lambda_0}$ is \emph{of type $H^i$} $(1<i<d-1)$ if, at each of its points, $\lambda_0$ is greater than other $i$ generalized Jacobi coordinates, and smaller than $d-i-1$ of them.

We say that a relativistic quadric placed on $\mathcal{Q}_{\lambda_0}$ is \emph{of type $0^i$} $(0<i<d-2)$ if, at each of its points, $\lambda_0$ is greater than other $i$ real generalized Jacobi coordinates, and smaller than $d-i-2$ of them.
\end{definition}

It would be interesting to analyze properties of the discriminant manifold $\Sigma_d$, as well as the combinatorial structure of the arrangement of relativistic quadrics, as it is done in Section 
\ref{sec:rel3} for $d=3$.
Remark that this description would have $[d/2]$ substantially different cases in each dimension, depending on choice of $k$ and $l$.

\begin{definition}
Suppose $(x_1,\dots,x_d)$ is a point of the $d$-dimensional Minkowski space $\mathbf{E}^{k,l}$ where equation (\ref{eq:jednacinad}) has real and different solutions.
\emph{Decorated Jacobi coordinates} of that point is the ordered $d$-tuplet of pairs:
$$
(E,\lambda_1),\quad
(H^1,\lambda_2),\quad
\dots,\quad
(H^{d-1},\lambda_d),
$$
of generalized Jacobi coordinates and the corresponding types of relativistic quadrics.
\end{definition}

Since we will consider billiard system within ellipsoids in the pseudo-Euclidean space, it is of interest to analyze behaviour of decorated Jacobi coordinates inside an ellipsoid.

\begin{proposition}\label{prop:svojstva12}
Let $\mathcal{E}$ be ellipsoid in $\mathbf{E}^{k,l}$ given by (\ref{eq:ellipsoid}).
We have:
\begin{itemize}
 \item[PE1]
each point inside $\mathcal{E}$ is the intersection of exactly $d$ quadrics from (\ref{eq:confocal});
moreover, all these quadrics are of different relativistic types;
 \item[PE2]
the types of these quadrics are $E$, $H^1$, \dots, $H^{d-1}$ -- each type corresponds to one of the disjoint intervals of the parameter $\lambda$:
$$(-a_d,-a_{d-1}),\ (-a_{d-1},-a_{d-2}),\ \dots,\ (-a_{k+1},0),\ (0,a_k),\ (a_k,a_{k-1}),\ \dots,\ (a_2,a_1).$$
\end{itemize}
\end{proposition}

\begin{proof}
The function given by the left-hand side of (\ref{eq:jednacinad}) is continous and strictly monotonous in each interval $(-a_d,-a_{d-1})$, $(-a_{d-1},-a_{d-2})$, \dots, $(-a_{k+2},-a_{k+1})$, $(a_k,a_{k-1})$, \dots, $(a_2,a_1)$ with infinite values at their endpoints.
Thus, equation (\ref{eq:jednacinad}) has one solution in each of them.
On the other hand, in $(-a_{k+1},a_k)$, the function is tending to $+\infty$ at the endpoints, and has only one extreme value -- the minimum.
Since the value of the function for $\lambda=0$ is less than $1$ for a point inside $\mathcal{E}$, it follows that equation (\ref{eq:jednacinad}) will have two solutions in $(-a_{k+1},a_k)$ -- one positive and one negative.
\end{proof}

In Proposition \ref{prop:svojstva12} we proved the relativistic analogs of properties E1, E2 from Section \ref{sec:euclid} for the Euclidean case.

\section{Billiards within quadrics and their periodic trajectories}
\label{sec:periodic}
In this section, we are going to derive first further properties of ellipsoidal billiards in the pseudo-Euclidean spaces.
In Section \ref{sec:billiards} we find in Theorem \ref{th:type} a simple and effective criterion for determining the type of a billiard trajectory, knowing its caustics.
Then we derive properties PE3--PE5 in Propostion \ref{prop:PE3-PE5}.
In Section \ref{sec:cayley} we prove the generalization of Poncelet theorem for ellipsoidal billiards in pseudo-Euclidean spaces and derive the corresponding Cayley-type conditions, giving a complete analytical description of periodic billiard trajectories in arbitrary dimension.
These results are contained in Theorems \ref{th:cayley} and \ref{th:poncelet}.

\subsection{Ellipsoidal billiards}
\label{sec:billiards}
\subsubsection*{Ellipsoidal billiard}

Billiard motion within an ellipsoid in the pseudo-Euclidean space is a motion
which is uniformly straightforward inside the ellipsoid, and obeys
the reflection law on the boundary.
Further, we will consider billiard motion within ellipsoid $\mathcal{E}$, given by equation
(\ref{eq:ellipsoid}) in $\mathbf{E}^{k,l}$.
The family of quadrics confocal with $\mathcal{E}$ is (\ref{eq:confocal}).

Since functions $F_i$ given by (\ref{eq:integralsF}) are integrals
of the billiard motion (see \cites{Mo1980,Audin1994,KhTab2009}), we
have that for each zero $\lambda$  of the equation (\ref{eq:discr}),
the corresponding quadric $\mathcal{Q}_{\lambda}$ is a caustic of
the billiard motion, i.e.\ it is tangent to each segment of the
billiard trajectory passing through the point $x$ with the velocity
vector $v$.

Note that, according to Theorem \ref{th:parametri.kaustike}, for a point placed inside $\mathcal{E}$, there are $d$ real solutions of equation (\ref{eq:jednacinad}). In other words, there are $d$ quadrics from the family (\ref{eq:confocal}) containing such a point, although some of them may be multiple. Also, by Proposition \ref{prop:kaustike} and Theorem \ref{th:parametri.kaustike}, a billiard trajectory within an ellipsoid will always have $d-1$ caustics.

According to Remark \ref{remark:type}, all segments of a billiard trajectory within $\mathcal{E}$ will be of the same type.
Now, we can apply the reasoning from Section \ref{sec:confocal} to billiard trajectories:

\begin{theorem}\label{th:type}
In the $d$-dimensional pseudo-Euclidean space $\mathbf{E}^{k,l}$, consider a billiard trajectory within ellipsoid $\mathcal{E}=\mathcal{Q}_0$, and let quadrics
$\mathcal{Q}_{\alpha_1}$, \dots, $\mathcal{Q}_{\alpha_{d-1}}$ from the family (\ref{eq:confocal}) be its caustics.
Then all billiard trajectories within $\mathcal{E}$ sharing the same caustics are of the same type: space-like, time-like, or light-like, as the initial trajectory.
Moreover, the type is determined as follows:
\begin{itemize}
\item
if $\infty\in\{\alpha_1,\dots,\alpha_{d-1}\}$, the trajectories are light-like;
\item
if $(-1)^l\cdot\alpha_1\cdot\dotsc\cdot\alpha_{d-1}>0$, the trajectories are space-like;
\item
if $(-1)^l\cdot\alpha_1\cdot\dotsc\cdot\alpha_{d-1}<0$, the trajectories are time-like.
\end{itemize}
\end{theorem}
\begin{proof}
Since values of functions $F_i$ given by (\ref{eq:integralsF}) are preserved by the billiard reflection and
$$
\sum_{i=1}^dF_i(x,v) = \langle{v,v}\rangle_{k,l},
$$
the type of the billiard trajectory depends on the sign of the sum
$\sum_{i=1}^dF_i(x,v)$.
From the equivalence of relations (\ref{eq:discr}) and (\ref{eq:polinom.P}), it follows that the sum depends only of the roots of $\mathcal{P}$, i.e.\ of parameters $\alpha_1$, \dots, $\alpha_{d-1}$ of the caustics.

Notice that the product $\alpha_1\cdot\dotsc\cdot\alpha_{d-1}$ is changed continuously on the variety of lines in $\mathbf{E}^{k,l}$ that intersect $\mathcal{E}$, with infinite singularities at light-like lines.
Besides, the subvariety of light-like lines divides the variety of all lines into subsets of space-like and time-like ones.
When passing through light-like lines, one of parameters $\alpha_i$ will pass through the infinity from positive to the negative part of the reals or vice versa;
thus, a change of sign of the product occurs simultaneously with a change of the type of line.

Now, take $\alpha_{j}=-a_{k+j}$ for $1\le j\le l$, and notice that all lines placed in the $k$-dimensional coordinate subspace $\mathbf{E}^k\times\mathbf{0}^l$ will have the corresponding degenerate caustics.
The reduced metrics is Euclidean in this subspace, thus such lines are space-like.
Since $\alpha_1$, \dots, $\alpha_k$ are positive for those lines of $\mathbf{E}^k\times\mathbf{0}^l$ that intersect $\mathcal{E}$, the statement is proved.
\end{proof}

Let us note that, in general, for the fixed $d-1$ quadrics from the confocal family, there can be found joint tangents of different types, which makes Theorem \ref{th:type} in a way unexpected. 
However, it turns out that, with fixed caustics, only lines having one type may have intersection with a given ellipsoid --- and only these lines give rise to billiard trajectories.

Next, we are going to investigate the behaviour of decorated Jacobi coordinates along ellipsoidal billiard trajectories.

\begin{proposition}\label{prop:PE3-PE5}
Let $\mathcal{T}$ be a trajectory of the billiard within ellipsoid $\mathcal{E}$ in pseudo-Euclidean space $\mathbf{E}^{k,l}$.
Denote by $\alpha_1$, \dots, $\alpha_{d-1}$ the parameters of the caustics from the confocal family (\ref{eq:confocal}) of $\mathcal{T}$, and take $b_1$, \dots, $b_p$, $c_1$, \dots, $c_q$ as in Theorem \ref{th:parametri.kaustike}.
Then we have:
\begin{itemize}
 \item[PE3]
along $\mathcal{T}$, each generalized Jacobi coordinate takes values in exactly one of the segments:
$$
[c_{2l-1},c_{2l-2}],\ \dots,\ [c_2,c_1],\ [c_1,0],\ [0,b_1],\ [b_2,b_3],\ \dots,\ [b_{2k-2},b_{2k-1}];
$$

 \item[PE4]
along $\mathcal{T}$, each generalized Jacobi coordinate can achieve local minima and maxima only at touching points with corresponding caustics, intersection points with corresponding coordinate hyper-planes, and at reflection points;

 \item[PE5]
values of generalized Jacobi coordinates at critical points are $0$, $b_1$, \dots, $b_{2k-1}$, $c_1$, \dots, $c_{2l-1}$;
between the critical points, the coordinates are changed monotonously.
\end{itemize}
\end{proposition}

\begin{proof}
Property PE3 follows from Theorem \ref{th:parametri.kaustike}.
Along each line, the generalized Jacobi coordinates are changed continuously.
Moreover, they are monotonous at all points where the line has a transversal intersection with a non-degenerate quadric.
Thus, critical points on a line are exactly touching points with corresponding caustics and intersection points with corresponding coordinate hyper-planes.

Note that reflection points of $\mathcal{T}$ are also points of transversal intersection with all quadrics containing those points, except with $\mathcal{E}$.
Thus, at such points, $0$ will be a critical value of the corresponding generalized Jacobi coordinate, and all other coordinates are monotonous.
This proves PE4 and PE5.
\end{proof}

The properties we obtained are pseudo-Euclidean analogs of properties E3--E5 from Section \ref{sec:euclid}, which are true for ellipsoidal billiards in Euclidean spaces.

\subsection{Analytic conditions for periodic trajectories}
\label{sec:cayley}
%\subsubsection*{Periodic trajectories}

Now, we are going to derive the corresponding analytic conditions of Cayley's type for periodic trajectories of the ellipsoidal billiard in the pseudo-Euclidean space, and therefore to obtain the generalization of the Poncelet theorem to pseudo-Euclidean spaces.

\begin{theorem}[Generalized Cayley-type conditions]\label{th:cayley}
In the pseudo-Euclidean space $\mathbf{E}^{k,l}$ ($k+l=d$), consider a billiard trajectory $\mathcal{T}$ within ellipsoid $\mathcal{E}$ given by equation (\ref{eq:ellipsoid}).
Let $\mathcal{Q}_{\alpha_1}$, \dots, $\mathcal{Q}_{\alpha_{d-1}}$ from confocal family (\ref{eq:confocal}) be caustics of $\mathcal{T}$.

Then $\mathcal{T}$ is periodic with period $n$ if and only if the following condition is satisfied:
$$
\rank\left(
\begin{array}{llll}
B_{d+1} & B_{d+2} & \dots & B_{m+1}\\
B_{d+2} & B_{d+3} & \dots & B_{m+2}\\
\dots & \dots & \dots & \dots\\
B_{d+m-1} & B_{d+m} & \dots & B_{2m-1}
\end{array}
\right) < m-d+1,\quad\text{for}\ n=2m;
$$
$$
\rank\left(
\begin{array}{llll}
B_{d} & B_{d+1} & \dots & B_{m+1}\\
B_{d+1} & B_{d+2} & \dots & B_{m+2}\\
\dots & \dots & \dots & \dots\\
B_{d+m-1} & B_{d+m} & \dots & B_{2m}
\end{array}
\right) < m-d+2,\quad\text{for}\ n=2m+1.
$$
Here,
$$
\sqrt{(\alpha_1-\lambda)\cdot\ldots\cdot(\alpha_{d-1}-\lambda)\cdot(a_1-\varepsilon_1\lambda)\cdot\ldots\cdot(a_d-\varepsilon_d\lambda)}
= B_0 + B_1\lambda + B_2\lambda^2 + \dots
$$
is the Taylor expansion around $\lambda=0$.
\end{theorem}

\begin{proof}
Denote:
$$
\mathcal{P}_1(\lambda)=(\alpha_1-\lambda)\cdot\ldots\cdot(\alpha_{d-1}-\lambda)\cdot(a_1-\varepsilon_1\lambda)\cdot\ldots\cdot(a_d-\varepsilon_d\lambda).
$$
Following Jacobi \cite{JacobiGW}, along a given billiard trajectory, we consider the integrals:
\begin{equation}\label{eq:int}
\sum_{s=1}^d\int\frac{d\lambda_s}{\sqrt{\mathcal{P}_1(\lambda_s)}},
\quad
\sum_{s=1}^d\int\frac{\lambda_{s}d\lambda_s}{\sqrt{\mathcal{P}_1(\lambda_s)}},
\quad\dots,\quad
\sum_{s=1}^d\int\frac{\lambda_{s}^{d-2}d\lambda_s}{\sqrt{\mathcal{P}_1(\lambda_s)}}.
\end{equation}
By PE3 of Proposition \ref{prop:PE3-PE5}, we may suppose that:
\begin{gather*}
\lambda_1\in[0,b_1],\ \lambda_i\in[b_{2i-2},b_{2i-1}]\ \text{for}\ 2\le i\le k;
 \\
\lambda_{k+1}\in[c_1,0],\ \lambda_{k+j}\in[c_{2j-1},c_{2j-2}]\ \text{for}\ 2\le j\le l.
\end{gather*}
Along a billiard trajectory, by PE4 and PE5 of Proposition \ref{prop:PE3-PE5}, each $\lambda_{s}$ will pass through the corresponding interval monotonously from one endpoint to another and vice versa alternately.
Notice also that values $b_1$, \dots, $b_{2k-1}$, $c_1$, \dots, $c_{2l-1}$ correspond to the Weierstrass points of hyper-elliptic curve: 
\begin{equation}\label{eq:curve}
\mu^2=\mathcal{P}_1(\lambda).
\end{equation}
Thus, calculating integrals (\ref{eq:int}), we get that the billiard trajectory is closed after $n$ reflections if and only if
$$
n\mathcal{A}(P_0)\equiv0
$$
on the Jacobian of curve (\ref{eq:curve}).
Here, $\mathcal{A}$ is the Abel-Jacobi map, and $P_0$ is the point on the curve corresponding to $\lambda=0$.
Further, in the same manner as in \cite{DragRadn1998b}, we obtain the conditions as stated in the theorem.
\end{proof}

As an immediate consequence, we get:

\begin{theorem}[Generalized Poncelet theorem]\label{th:poncelet}
In pseudo-Euclidean space $\mathbf{E}^{k,l}$ ($k+l=d$), consider a billiard trajectory $\mathcal{T}$ within ellipsoid $\mathcal{E}$.

If $\mathcal{T}$ is periodic and become closed after $n$ reflections on the ellipsoid, then any other trajectory within $\mathcal{E}$ having the same caustics as $\mathcal{T}$ is also periodic with period $n$.
\end{theorem}

\begin{remark}
The generalization of the Full Poncelet theorem from \cite{CCS1993} to pseudo-Euclidean spaces is obtained in \cite{WFSWZZ2009}.
However, only space-like and time-like trajectories were discussed there.

A Poncelet-type theorem for light-like geodesics on the ellipsoid in the three-dimensional Minkowski space is proved in \cite{GKT2007}.
\end{remark}

\begin{remark}
Theorems \ref{th:cayley} and \ref{th:poncelet} will also hold in symmetric and degenerated cases, that is when some of the parameters $\varepsilon_i a_i$, $\alpha_{j}$ concide, or in the case of light-like trajectories, when $\infty\in\{\alpha_{j}\mid1\le j\le d-1\}$.
In such cases, we need to apply the desingularisation of the corresponding curve, as explained in detail in our works \cite{DragRadn2006,DragRadn2008}.

When we consider light-like trajectories, then the factor containing the infinite parameter is ommited from polynomial $\mathcal{P}_1$.
\end{remark}

\begin{example}
Let us find all $4$-periodic trajectories within ellipse $\mathcal{E}$ given by
(\ref{eq:ellipse}) in the Minkowski plane, i.e.~ all conics $\mathcal{C}_{\alpha}$ from the confocal family (\ref{eq:confocal.conics}) corresponding to such trajectories.

By Theorem \ref{th:cayley}, the condition is $B_3=0$, with
$$
\sqrt{(a-\lambda)(b+\lambda)(\alpha-\lambda)}=B_0+B_1\lambda+B_2\lambda^2+B_3\lambda^3+\dots
$$
being the Taylor expansion around $\lambda=0$.
Since
$$
B_3=\frac{(-ab-a\alpha+b\alpha)(-ab+a\alpha+b\alpha)(ab+a\alpha+b\alpha)}
{16(ab\alpha)^{5/2}},
$$
we obtain the following solutions:
$$
\alpha_1=\frac{ab}{b-a},\quad
\alpha_2=\frac{ab}{a+b},\quad
\alpha_3=-\frac{ab}{a+b}.
$$
Since $\alpha_1\not\in(-b,a)$ and $\alpha_2,\alpha_3\in(-b,a)$, conic
$\mathcal{C}_{\alpha_1}$ is a hyperbola, while $\mathcal{C}_{\alpha_2}$,
$\mathcal{C}_{\alpha_3}$ are ellipses.
\end{example}

%\section{Conclusion}
%\label{sec:conc}
%\input{6-conc/conc}

%\subsection*{Acknowledgement}

%\section{Pseudo-Euclidean discrete Heisenberg system}
%\label{sec:heisenberg}

\end{document}